\documentclass[10pt]{amsart}

\usepackage[utf8]{inputenc}
\usepackage{amsmath,amssymb,stmaryrd,amsthm,mathtools}
\mathtoolsset{showonlyrefs}

\usepackage{tikz}
\usepackage{tikz-cd}
\usepackage[all]{xy}
\usepackage[multiple]{footmisc}
\usepackage{bbm}

\theoremstyle{plain}
\newtheorem{theorem}{Theorem}[section]

\newtheorem*{unnumberedTheorem}{Theorem}
\newtheorem{lemma}[theorem]{Lemma}
\newtheorem{proposition}[theorem]{Proposition}

\newtheorem{corollary}[theorem]{Corollary}
\theoremstyle{definition}
\newtheorem{remark}[theorem]{Remark}
\newtheorem{definition}[theorem]{Definition}
\newtheorem{rappels}[theorem]{Rappels}
\newtheorem{notation}[theorem]{Notation}

\newtheorem{setting}[theorem]{Setting}
\usepackage{calrsfs} 
\DeclareMathAlphabet{\pazocal}{OMS}{zplm}{m}{n}
   \newcommand{\dagg}{{\normalfont({\boldmath$\dagger$})}}

\newcommand{\leswithmaps}[5]{
	\begin{tikzcd}[column sep= 0.5cm]
		\pgfmatrixnextcell#1\arrow[r,"#2"] \pgfmatrixnextcell#3 \arrow[r,"#4"] \pgfmatrixnextcell#5
	\end{tikzcd}
}

\newcommand{\les}[3]{\leswithmaps{#1}{}{#2}{}{#3}}
\newcommand{\disttrianglewithmaps}[5]{
	\begin{tikzcd}[column sep= 0.6cm]
	#1\arrow[r,"#2"] \pgfmatrixnextcell #3 \arrow[r,"#4"] \pgfmatrixnextcell#5\arrow[r,"+1"]\pgfmatrixnextcell ~
	\end{tikzcd}
}
\newcommand{\disttriangle}[3]{\disttrianglewithmaps{#1}{}{#2}{}{#3}}

\newcommand{\cartesiandiagramwithmaps}[8]{
\begin{tikzcd} #1  \arrow[r,"#2"]  \arrow[d,"#4"']   \pgfmatrixnextcell  #3  \arrow[d,"#5"]  \\ #6  \arrow [r,"#7"'] \pgfmatrixnextcell  #8 
\end{tikzcd}
}

\newcommand{\setline}{\,|\,}

\newcommand{\un}{\mathbbm{1}}
\newcommand{\kk}{\mathbbm{k}}

\newcommand{\A}{\mathbb A}

\newcommand{\C}{\mathbb C}
\newcommand{\Chow}{CH}
\newcommand{\cato}{\pazocal O}
\newcommand{\Dual}{\operatorname{D}}
\newcommand{\Ee}{\pazocal E}
\newcommand{\F}{\mathbb F}
\newcommand{\Ff}{\pazocal F}
\newcommand{\G}{\mathbb G}

\newcommand{\g}{\mathfrak{g}}
\renewcommand{\H}{\mathsf{H}}

\newcommand{\Hyp}{\mathbb H}
\newcommand{\h}{\mathfrak{h}}

\newcommand{\uK}{\underline{K}}
\newcommand{\K}{\mathbb{K}}

\newcommand{\Mm}{\pazocal M}
\newcommand{\Nn}{\pazocal N}

\newcommand{\OO}{\pazocal{O}}
\newcommand{\op}{\operatorname{op}}
\newcommand{\Pcal}{\mathcal{P}}
\newcommand{\Proj}{\mathbb P}
\renewcommand{\P}{\mathbb P}
\newcommand{\Projectives}{\mathsf{Proj}\,}

\newcommand{\Tilt}{\mathsf{Tilt}\,}
\newcommand{\R}{\mathbb{R}}
\newcommand{\simple}{\pazocal S}
\newcommand{\Sch}{\mathsf{Sch}}

\newcommand{\SymA}{\operatorname{S}}
\newcommand{\Spec}{\mathsf{Spec}}
\newcommand{\Scal}{\mathcal{S}}
\newcommand{\Ss}{\pazocal S}
\newcommand{\Simp}{S}
\newcommand{\T}{\mathbb T}
\newcommand{\Tt}{\pazocal T}
\newcommand{\V}{\mathbb V}

\newcommand{\Weyl}{W}
\newcommand{\Z}{\mathbb Z}
\newcommand{\sun}{\un^\Sfrak}
\newcommand{\sotimes}{%
\mathchoice%
{\stackrel{\Sfrak}{\otimes}}
{\otimes^\Sfrak}%
{\otimes^\Sfrak}%
{\otimes^\Sfrak}%
}%
\newcommand{\rotimes}{%
\mathchoice%
{\stackrel{\R}{\otimes}}
{\otimes^\R}%
{\otimes^\R}%
{\otimes^\R}%
}%
\newcommand{\QProj}{\mathsf{QProj}}
\newcommand{\Sm}{\mathsf{Sm}}
\newcommand{\PSh}{\mathsf{PSh}}
\newcommand{\Sh}{\mathsf{Sh}}
\newcommand{\Shv}{\mathsf{Sh}}

\newcommand{\Nis}{\mathsf{Nis}}

\newcommand{\Ind}{\mathsf{Ind}}

\renewcommand{\mod}{\mathsf{-mod}}
\newcommand{\id}{\mathsf{id}}

\renewcommand{\ker}{\mathsf{ker}}
\newcommand{\coker}{\mathsf{coker}}
\newcommand{\Sfrak}{\mathfrak{S}}
\newcommand{\Cone}{\mathsf{Cone}}

\newcommand{\Cat}{\mathsf{Cat}}

\newcommand{\HZop}[1][{}]{\mathsf{H}\Z\!/\!p}
\newcommand{\Zop}[1][{}]{\Z\!/\!p}
\newcommand{\Forget}{\mathsf{Res}}
\newcommand{\Htp}{\mathsf{Htp}}
\newcommand{\Stb}{\mathsf{Stb}}
\newcommand{\kmp}[1][{}]{\uK^M_{#1}\!\!/p}

\newcommand{\Der}{\operatorname{Der}}
\newcommand{\modules}{\operatorname{-mod}}
\newcommand{\soergelmodules}{\operatorname{-Smod}}

\newcommand{\Derb}[1]{\operatorname{Der^b}(#1)} 
\newcommand{\Hotb}[1]{\operatorname{Hot^b}(#1)} 
\newcommand{\DT}[1]{\operatorname{MTDer}(#1)} 
\newcommand{\DMT}[2]{\operatorname{MTDer}_{#1}(#2)} 
\newcommand{\Par}[2]{\operatorname{Par}_{#1}(#2)} 
\newcommand{\Per}[2]{\operatorname{Per}_{#1}(#2)} 

\newcommand{\Hotub}{\operatorname{Hot}}
\newcommand{\Hom}[1]{\operatorname{Hom}_{#1}} 
\newcommand{\iHom}[1]{\mathcal{H}\hspace{-1.9pt}om_{#1}} 
\newcommand{\End}[1]{\operatorname{End}_{#1}} 
\newcommand{\BS}{\operatorname{BS}} 

\begin{document}
	\title[Motives and Representations]{Mixed Motives and Geometric Representation Theory in Equal Characteristic}

\author{Jens Niklas Eberhardt}
\email{jneberhardt@gmail.com}  
\address{Department of Mathematics, University of California Los Angeles\\520 Portola Plaza, Los Angeles, CA 90095}
\author{Shane Kelly}
\email{shane.kelly.uni@gmail.com}  
\address{Department of Mathematics, Tokyo Institute of Technology\\ 2-12-1 Ookayama, Meguro-ku,
Tokyo 152-8551, Japan}
\subjclass{14F05 (Primary), 14C15, 14M15, 19D45, 20G40 (Secondary)}
\keywords{motives, representation theory, positive characteristic, reductive groups}
\begin{abstract}
	Let $\kk$ be a field of characteristic $p$. We introduce a formalism of mixed sheaves with coefficients in $\kk$ and apply it in representation theory.\\
	We construct a system of $\kk$-linear triangulated category of motives on schemes over $\overline{\mathbb{F}}_p$, which has a six functor formalism and computes higher Chow groups. Indeed, it behaves similarly to other categories of mixed sheaves that one is used to. We attempt to make its construction also accessible to non-experts.
	\\
	Next, we consider the subcategory of \emph{stratified mixed Tate motives} defined for affinely stratified varieties, discuss perverse and parity motives and prove formality results. We combine this with results of Soergel to construct a geometric and graded version of the derived \emph{modular category} $\cato(G)$, consisting of rational representations of a semisimple algebraic group $G/\kk$.
\end{abstract}
\maketitle
\section{Introduction}
\subsection{Mixed Sheaves}
Categories of mixed $\ell$-adic sheaves and mixed Hodge modules are indispensable tools in geometric representation theory. They are used in the proof of the Kazhdan--Lusztig conjecture, uncover hidden gradings in categories of representations \cite{BGS} or categorify objects as Hecke algebras \cite{Springer}, representations of quantum groups \cite{lusztig2010introduction} and link invariants \cite{webster2017}, to name a few. But they are---by their nature---limited to \emph{characteristic zero} coefficients. 

In this paper, we address this gap and propose a formalism of mixed sheaves with coefficients in \emph{characteristic $p$} by following the idea of Soergel and Wendt \cite{SoeWe} to make use of the recent developments in the world of motivic sheaves. Most importantly, our formalism comes equipped with six functors and computes Chow groups, which is often all one really needs in the specific applications. It allows us to translate many results from the characteristic zero setting to characteristic $p$, as we will show.

\subsection{Motives}
Just as constructible sheaves live inside more general triangulated categories, as for example étale sheaves, our formalism of mixed sheaves is developed using \emph{mixed motives in equal characteristic}, i.e., motives on characteristic $p$ schemes with characteristic $p$ coefficients. Using the work of Ayoub \cite{Ayo07}, Cisinski--Déglise \cite{CD} and Geisser--Levine \cite{GL} we will show:
\begin{unnumberedTheorem}[Theorem \ref{theo:propertiesofH}, Corollary \ref{coro:calculateHigherChowGroups}]
There is a system $\H(X,\kk)$ of $\kk$-linear tensor triangulated categories of motives associated to quasi-projective schemes $X/\overline\F_p$. It is equipped
with a six functor formalism fulfilling all the usual properties. Moreover for smooth varieties $X\rightarrow \overline\F_p$ one has
	\[ \Chow^{n}(X, 2n{-}i; \kk) \cong \Hom{\H(X,\kk)}(\un_X, \un_X(n)[i]) \]
	where the left hand side denotes higher Chow groups.
\end{unnumberedTheorem}

The category $\H(X,\kk)$ is essentially the \emph{homotopy category of modules over the $T$-spectrum representing motivic cohomology with $\kk$-coefficients in the Morel--Voevodsky stable homotopy category}, sometimes written as $\H\kk_X\modules$ in the motivic literature. 
However, we have attempted to construct it in a way which is as accessible as possible to non-homotopy theorists. Using the elementary Milnor $K$-theory we avoid any discussion of presheaves with transfers, and using techniques from \cite{CD} we avoid any mention of simplicial sets or $S^1$-spectra, and indeed, even manage to avoid the word $T$-spectrum. We build our categories step-by-step, from the ground up, using honest sheaves, modules over an explicit monoïd object, derived categories of abelian categories, and Verdier quotients, inviting the non-expert to take a peek inside the black box.
\subsection{Stratified Mixed Tate Motives} Following \cite{SoeWe}, we define for affinely stratified varieties $X$ the category of \emph{stratified mixed Tate motives} $\DMT{\Ss}{X,\kk}$ as a full subcategory of $\H(X,\kk)$ consisting of motives which restrict to finite direct sums of Tate objects $\un(n)[m]$ on each stratum. This is the analogue of constructible sheaves whose cohomology sheaves are locally constant on each stratum.

Stratified mixed Tate motives behave well under the six functors and admit a perverse $t$-structure and weight structure. Furthermore, the theory of parity sheaves of Juteau--Mautner--Williamson \cite{JMW} applies, and we will be able to show:
\begin{unnumberedTheorem}[Corollary \ref{cor:paritymotivestilting}, Theorem \ref{thm:perversemotivestilting}]\label{thm:intro2}
	Let $(X,\Ss)$ be an affinely stratified variety over $\overline\F_p$  fulfilling some additional conditions---all of them are fulfilled for flag varieties with their Bruhat-stratification. Then there are equivalences of categories
\begin{equation*}\label{eq:introtheorem3}
\Derb{\Per{\Ss}{X,\kk}}
\stackrel{\sim}{\leftarrow}
\DMT{\Ss}{X,\kk}
\stackrel{\sim}{\rightarrow}
\Hotb{\Par{S}{X,\kk}_{w=0}}
\end{equation*}
between the derived category of perverse motives, the category of stratified mixed Tate motives, and the homotopy category of weight zero parity motives on $X$.
\end{unnumberedTheorem}
The proof heavily relies on the fact that there are no non-trivial extensions between the Tate objects $\un(n)$ in $\H(\Spec(\overline{\F}_p),\kk)$, which boils down to a classical observation of Steinberg, namely that the Milnor $K$-groups $K_n^M(\F_{p^m})/p$ vanish for $n>0$. 
This is the reason why, of all things, we work with mixed motives in \emph{equal characteristic}.
\subsection{Representation Theory} We conclude the paper with a particular representation theoretic application of our formalism.
Let $G/\kk$ be a semisimple simply connected split algebraic group, for example $\operatorname{SL}_n/{\F_q}$. A fundamental problem in representation theory is to determine the characters of all simple rational $G$-modules. Unlike in the characteristic zero case, this is still wide open and the subject of ongoing research. In \cite{Soe00}, Soergel proposes a strategy using geometric methods: He translates the problem---at least for some of the simple modules---into a question about the geometry of 
a flag variety $X^\vee$. He does this by relating \emph{parity sheaves} on $X^\vee(\C)^{an}$ to the projective objects in the \emph{modular category $\cato(G)$} (see Definition~\ref{defi:modO}), a subquotient of the category of $G$-modules. But the beauty and clarity of these and other results in characteristic $p$ geometric representation theory suffered---so far---
from the lack of an appropriate formalism of \emph{mixed sheaves} with coefficients in $\kk$. We will show:
\begin{unnumberedTheorem}[
Theorem~\ref{thm:modularcategoryo}]\label{thm:intro3}
Let $G/\kk$ be a semisimple simply connected split algebraic group and $X^\vee/\overline{\F}_p$ be the flag variety of the Langlands dual group. Then there is an equivalence of categories
\begin{equation*}
\DMT{(B^\vee)}{X^\vee,\kk}
\stackrel{\sim}{\to}
\Derb{\cato^{\Z,ev}(G)}
\end{equation*}
between the category of stratified mixed Tate motives on 
$X^\vee$ 
and the derived evenly graded modular category $\cato^{\Z,ev}(G)$. We have to assume that $p$ is bigger than the Coxeter number of $G$.
\end{unnumberedTheorem}
This equips the modular category $\cato(G)$ with all the amenities of the geometric world, as for example a full six functor formalism. 
\subsection{Future Work}
\begin{enumerate}
		\item  In a future paper we extend our results to the equivariant setting and to ind-schemes. Here Iwahori constructible stratified mixed Tate motives on the affine Grassmannian provide a graded version of the derived principal block of the category of all rational representations of $G$, this is the \emph{graded Finkelberg-Mirkovi\'c
 conjecture}, see \cite{AR16}. 
		\item For quasi-projective schemes $X/\Z$, we strive to construct a realization functor
		$$real: \H(X/\overline{\F}_p,\kk)\rightarrow \Der(X(\C)^{an}, \kk)$$
		into the derived category of constructible sheaves, compatible with the six functors, see \cite[Section 17]{CD}. This would allow us to prove the \emph{ungraded} Finkelberg--Mirkovi\'c
 conjecture, which is still an open conjecture.
		\item Our formalism should also be useful in categorification. In \cite{webster2017} Webster and Williamson give a geometric construction of the triply graded Khovanov--Rozansky link homology (a categorifaction of the HOMFLYPT polynomial) using mixed $\ell$-adic sheaves on flag varieties. This immediately translates into our setting, and allows one to do the same with coefficients in characteristic $p$. Since in our settings Tate objects do not extend, their techniques can even be simplified.
		\item It would be interesting to study the action of motivic cohomology operations, in particular the Steenrod algebra, on stratified mixed Tate motives with mod-$p$ coefficients. In the case of the flag variety one should obtain a straightforward description of the action on Soergel modules, since they are just constructed from copies of the equivariant motivic cohomology ring $H_{\mathcal M,\G_m}^{\bullet}(\Spec(\overline{\F}_p),\kk(\bullet))=\kk[u]$, where the Steenrod reduced powers just act by $\pazocal P^i(u^k)={k \choose i} u^{k+i(p-1)}.$
		\item Our construction of  $\H(X, \Z/p)$ works \textit{mutatis mutandis} with $\Z/p^n$ coefficients, which allows one to consider $\H(X, \mathbb{Q}_p)=\varprojlim\H(X, \Z/p^n)\otimes_{\Z_p}\mathbb{Q}_p$. 
\end{enumerate}
\subsection{Relation to other work.}
\begin{enumerate}
 \item First and foremost, it must be said that this project is strongly inspired by Soergel and Wendt's work in characteristic zero, \cite{SoeWe}.

		\item The statements in Equations \ref{eq:introtheorem3} and \ref{eq:introdmtishotweightzero}  can also be interpreted as a \emph{formality} result. Namely, that $\DMT{\Ss}{X}$ can be realized as category of dg-modules over a formal dg-algebra. Similar formality results for the flag variety were achieved first by \cite{RSW14} using étale sheaves on $X^\vee/\F_p$ with mod-$\ell$ coefficients, but with a stronger requirement on $\ell$, and then by Achar and Riche using their \emph{mixed derived category} (see the next point).

		\item In  \cite{AR1} and \cite{AR2}, Achar and Riche present another approach to mixed sheaves. They have the ingenious idea to simply \emph{define} their mixed derived category of a stratified variety $X/\C$ to be the homotopy category of parity sheaves on its complex points $X(\C)^{an}$, equipped with the metric topology. Equations \ref{eq:introtheorem3} and \ref{eq:introdmtishotweightzero} imply that this coincides with our category of stratified mixed Tate motives, at least when $X$ is the flag variety---here weight zero parity motives on $X/\overline{\F}_p$ coincide with parity sheaves on $X(\C)$.
		They also reverse engineer some parts of a six functor formalism, as for example base change for locally closed embeddings.
		On the other hand, our approach has the advantage of being embedded in the rich environment of motives, which immediately implies all those properties and also provides new structures, as for example an action of the motivic cohomology operations. It even yields a sensible outcome for a variety which does not have enough parity motives.
	
	\item The systems of categories $\H\kk_X\modules$ are of great interest, and nailing down the six functor formalism à la \cite{Ayo07} for them (for a general ring $\kk$) is one of the achievements of \cite{CD}. Our contribution is to observe that in our case, Milnor $K$-theory gives a particularly nice model of $\H\kk_X$ allowing one to largely avoid the abstract homotopy theory.
	\end{enumerate}
	\subsection{Outline} In Section~\ref{sec:motives} we construct the system of categories $\H(X)$ of motives with coefficients in a commutative $\Zop[p]$-algebra $A$ equipped with a six functor formalism. As mentioned before, we go out of our way to make this as accessible as possible to the non-expert. A roadmap to our construction can be found in Section~\ref{sec:constructionOverview}.

In Section~\ref{sec:mstm} we define the category of stratified mixed Tate motives as a full subcategory
\begin{equation*}
\DMT{\Ss}{X,\kk}  \subseteq H(X, \kk).
\end{equation*}
We consider a weight structure on this category, and prove that $\DMT{\Ss}{X,\kk}$ is equivalent to the bounded homotopy category of its weight zero objects 
\begin{equation*}\label{eq:introdmtishotweightzero}
\DMT{\Ss}{X}\cong\Hotb{\DMT{\Ss}{X}_{w=0}},
\end{equation*}
Theorem~\ref{thm:tilting}. This shows that $\DMT{\Ss}{X}$ is the dg-derived category of a formal (equipped with a trivial differential) graded dg-algebra, Theorem~\ref{theo:formality}. In Section~\ref{subsec:Erweiterungssatz} we state the Erweiterungssatz, Theorem~\ref{thm:Erweiterungssatz}, which implies that the weight zero motives in $\DMT{\Ss}{X}_{w=0}$ can be realised as graded modules (Soergel modules) over the Chow ring of $X$, Corollary~\ref{coro:pwPureErweiterungssatz}
\begin{equation*}
\DMT{\Ss}{X}_{w=0} \subseteq \Chow^\bullet(X,\kk(\bullet))\modules^{\Z}.
\end{equation*}

In Section~\ref{sec:parityMotives} we study certain interesting subcategories of our category of stratified mixed Tate motives---\emph{parity motives} in Section~\ref{sec:parity} and \emph{perverse motives} in Section~\ref{subsec:perversemotives}. %
The general principle here is that everything works as one is used to from constructible étale sheaves or mixed Hodge modules. In Section~\ref{subsec:flagvar} we consider the case of flag varieties. 

In Section~\ref{sec:repTheory} we apply our results to the representation theory of semisimple algebraic groups in equal characteristic. This is where we observe that we can obtain $\OO(G)$ from our categories, Theorem~\ref{thm:modularcategoryo}.

In Section~\ref{sec:Cat} we recall some notions from category theory needed in the construction of $\H(X,\kk)$ for the convenience of the reader.
\\
\\
\emph{Acknowledgements.} We thank the referee for their very detailed and constructive suggestions improving the exposition of this article substantially.  We would like to thank Wolfgang Soergel for many encouraging and illuminating discussions.
The first author thanks Markus Spitzweck and Matthias Wendt for instructive email exchanges about motivic six functors.
We also thank Oliver Braunling, Brad Drew, Thomas Geisser, Lars Thorge Jensen and Simon Pepin Lehalleur for their valuable input. 
The second author thanks the first author for the opportunity to work on this project, and for many stimulating questions and discussions. %
The first author was financially supported by the DFG Graduiertenkolleg 1821 ``Cohomological Methods in Geometry''.
%

\section{Categories of motives and the six operations} \label{sec:motives}

Everything in this section is over an arbitrary perfect base field $k$ of characteristic $p$, except for Corollary~\ref{coro:fieldChowGroups} where we restrict to an algebraic extension of $\F_p$. In fact, everything preceding Section~\ref{sec:motivesChow} works with $k$ replaced by an arbitrary separated noetherian base scheme.

In this section we construct a system of categories of motives with coefficients in a commutative $\Zop[p]$-algebra $A$ equipped with a six functor formalism. %
More explicitly, in Definition~\ref{defi:H} we associate to every quasi-projective variety $X/k$ (resp. morphism of quasi-projective varieties $f: Y \to X$) a symmetric monoïdal triangulated category (resp. a tensor triangulated functor) %
\begin{equation*}
 \H(X)=\H(X, A), 
\qquad \textrm{ resp. } \qquad 
f^*: \H(Y) {\to} \H(X).
\end{equation*}

Using a theorem of Geisser--Levine we show in Corollary~\ref{coro:calculateHigherChowGroups} that when $X$ is smooth, the category $\H(X)$ can be used to calculate motivic cohomology in Voevodsky's sense, or equivalently, Bloch's higher Chow groups, \cite[Theorem 19.1]{MVW}, in the sense that there are canonical isomorphisms, 
\begin{align*}
\Hom{\H(X,A)}(\un, \un(i)[j]) 
\ \cong \ H^{j}_\mathcal{M}(X, A (i)) %
\ \cong \ \Chow^i(X, 2i-j; A).
\end{align*}

We observe in Corollary~\ref{coro:stableHomotopy2functor} that this system of categories is what Ayoub calls a unital symmetric monoïdal stable homotopy 2-functor, \cite[Definitions 1.4.1 and 2.3.1]{Ayo07}, and consequently, satisfies the following list of properties.
\begin{theorem} \label{theo:propertiesofH} Let $k$ be a perfect field (or more generally a separated noetherian scheme). 
\begin{enumerate}
 \item For every morphism $f: Y \to X$ in $\QProj/k$ the functor $f^*$ has a right adjoint, \cite[Def.1.4.1]{Ayo07}.
 \begin{equation}
 f^* : \H(X) \rightleftarrows \H(Y): f_*. 
\end{equation}

\item \label{theo:propertiesofH:except} For any  morphism $f:Y\rightarrow X$ in $\QProj/k$, one can construct a further pair of adjoint functors, the {\bf exceptional functors}  
$$
f_!:\H(Y)\leftrightarrows
\H(X):f^! 
$$
which fit together to form a covariant (resp. contravariant) 
$2$-functor  $f\mapsto f_!$ (resp. $f\mapsto f^!$), \cite[Prop.1.6.46]{Ayo07}. 

 \item For each $X \in \QProj/k$, the tensor structure on $\H(X)$ is closed in the sense that for every $E \in \H(X)$, the functor $-\otimes E$ has a right adjoint 
 \begin{equation}
 - \otimes E : \H(X) \rightleftarrows \H(X): \iHom{X}(E, -), 
\end{equation}

the \textbf{internal Hom functor}.

 \item \label{theo:propertiesofH:stability} \textbf{(Stability)} %
  For every $X \in \QProj/k$, let $p: \A^1_X \to X$ be the canonical projection with zero section $s$. Then the endofunctor
 \begin{equation}
 s^!p^* : \H(X) \to \H(X) 
\end{equation}
is invertible, \cite[Def.1.4.1]{Ayo07}. For $E\in \H(X)$ and $n\in \Z$ we denote
\begin{equation}
 E(n):=(s^!p^*)^{n}(E)[-2n] 
\end{equation}

the $n$th \textbf{Tate twist} of $E$.

 \item \label{theo:propertiesofH:homotopy}With $X$ and $p$ as above, $\H$ satisfies \textbf{$\A^1$-homotopy invariance} in the sense that the unit of the adjunction $(p^*, p_*)$ is an isomorphism, \cite[Def.1.4.1]{Ayo07}.
 \begin{equation}
 \id \stackrel{\sim}{\to} p_*p^*. 
\end{equation}

 \item \label{theo:propertiesofH:properRightAdjoint}For any $f:Y\rightarrow X$ in
 $\QProj/k$ there exists a natural transformation \[f_!\rightarrow
f_\ast\] which is an isomorphism when $f$ is proper, \cite[Def.1.7.1, Thm.1.7.17]{Ayo07}.

 \item (\textbf{Relative purity}) %
  For any smooth morphism $f:Y\rightarrow X$ in
 $\QProj/k$ of relative dimension $d$ there is a canonical isomorphism, \cite[\S 1.5.3]{Ayo07}, \begin{equation}
f^* \rightarrow f^!(-d)[-2d].
\end{equation}

\item (\textbf{Base change}) For any cartesian square 
\begin{center}
  \begin{minipage}[c]{10cm}
    \xymatrix{
      X'\ar[r]^{g'} \ar[d]_{f'} & X\ar[d]^f \\
      Y' \ar[r]_g & Y
    }
  \end{minipage}
\end{center}
there exist natural isomorphisms of functors, \cite[Prop.1.6.48, Chap.1]{Ayo07},
$$g^\ast f_!\stackrel{\sim}{\longrightarrow}f_!'g'^\ast,\qquad
g'_\ast f'^!\stackrel{\sim}{\longrightarrow}f^!g_\ast,
$$ 
\item \label{theo:propertiesofH:localisation} ({\bf Localization}),
  For $i:Z\rightarrow X$ a closed immersion with open complement $j:U\to 
X$, there are distinguished triangles
$$
j_!j^!\rightarrow 1\rightarrow i_\ast i^\ast \rightarrow j_!j^![1]
$$
$$
i_! i^!\to 1\to j_\ast j^\ast \to i_! i^! [1]
$$
where the first and second maps are the counits and units  of the
respective adjunctions, \cite[Lem.1.4.6, 1.4.9]{Ayo07}.
\item 
 (\textbf{Projection formulae, Verdier duality}) For any morphism  $f:Y\rightarrow X$ in
  $\QProj/k$, there exist natural isomorphisms 
\begin{align*}
(f_!E) \otimes_XF&\stackrel{\sim}{\longrightarrow}
f_!(E\otimes_Y f^\ast F),\\
\iHom{X}(E,f_\ast F)&\stackrel{\sim}{\longrightarrow} f_\ast \iHom{Y}(f^\ast E,F),\\
\iHom{X}(f_!E,F)&\stackrel{\sim}{\longrightarrow}
f_\ast \iHom{Y}(E,f^!F),\\
f^!\iHom{X}(E,F)&\stackrel{\sim}{\longrightarrow}
\iHom{Y}(f^\ast E,f^!F).
\end{align*}

\item 
Define the subcategory of \textbf{constructible} objects $\H^c(S)\subset \H(S)$ to be the subcategory of compact objects. This subcategory coincides with the thick full subcategory generated by $f_!f^!\un(n)$ for $n\in\mathbb{Z}$ and $f: X\to S$ smooth. The six functors $f_!, f^!, f^*, f_*, \otimes, \iHom{}$ preserve compact objects.%

\item \label{theo:propertiesofH:dual} Let $f:X\to \Spec(k)$ in $\QProj/k$. For $E\in\H(X)$ we denote by 
\begin{equation} \label{equa:propertiesofH:dual}
\Dual_X(E):=\iHom{X}(E,f^!(\un))
\end{equation}
 the \textbf{Verdier dual} of $E$. 
For  all $E,  F\in \H^c(X)$, there is a canonical duality isomorphism  
$$
\Dual_X(E\otimes \Dual_X(F))\stackrel{\sim}{\rightarrow} \iHom{X}(E,F).
$$ 
Furthermore, for any morphism $f:Y\rightarrow X$ in $\QProj/k$ and any $E\in\H^c(X)$ there are natural isomorphisms  
\begin{align*}
\Dual_X(\Dual_X(E))&\cong E,\\
\Dual_Y(f^\ast(E))&\cong f^!(\Dual_X(E)),\\
\Dual_X(f_!(E))&\cong f_\ast(\Dual_Y(E)).
\end{align*}
\end{enumerate}
\end{theorem}
Finally, let us mention that $\H(-,A)$ is canonically equipped with a morphism from the constant stable homotopy functor with value $D(A\modules)$ the derived category of $A$-modules. More explicitly, for every $X \in \QProj/k$ there is a canonical \emph{tensor triangulated} functor, $\gamma$, which admits a right adjoint 
\begin{equation}
 \gamma: D(A\modules) \rightleftarrows \H(X) : \Gamma 
\end{equation}
compatible with the functor $f^*$.
So in particular, we get formulas like
\begin{equation}
 H^n \Gamma(\iHom{X}(E,F)) = \Hom{\H(X)}(E, F[n]). 
\end{equation}

\subsection{Construction motivation and overview} \label{sec:constructionOverview}

Philosophically, categories of motives satisfy a universal property; they are universal targets for functors from schemes satisfying certain properties (of course, which schemes and which properties depends on the task at hand). As such, many of the constructions, including the one presented here, have a generators-and-relations flavor.

We want our category of motives amongst other things to fulfill all the properties listed in Theorem \ref{theo:propertiesofH}. To achieve this, it suffices to implement a short list of ``axioms'' (existence of right adjoints $f_*$ in general, and left adjoints $f_\#$ for smooth morphisms, invertibility of $M(\G_m / 1)$, Nisnevich descent, $\A^1$-invariance). The fact that this gives rise to all six functors and their properties, in particular purity and duality, is difficult, and due to Morel--Voevodsky, and Ayoub.

We will end up with a sequence of \emph{morphisms of monoïdal $\Sm$-fibered categories} (i.e., categories equipped with $f_\#, f^*, \otimes$, see Definition~\ref{definition:monoidalsmfibered}),
%

\begin{center}
		\begin{tikzcd}[column sep= 0.9cm, row sep= 0.8cm]
		\Sm / S \arrow[rd,"\Z_\Nis(-)"]\arrow[r,"\Z(-)"] \arrow[rrdd,"\K(-)"',bend right=12] \arrow[dd,"M(-)"'] 
		& \PSh(\Sm / S) \arrow[d]  \rar & \PSh(\Sm / S)^\Sfrak \arrow[d]\\
		
		& \Sh_\Nis(\Sm / S)  \rar & \Sh_\Nis(\Sm / S)^\Sfrak \arrow[d,"\K \sotimes -"]
	 \\
		\H(S)\stackrel{def}{=}D(\K_S\modules) / \Ss_{\Htp, \Stb}
		& D(\K_S\modules) \lar 
		& \K_S\modules 
		\lar
	\end{tikzcd}
\end{center}
of which all but $\Sm / S$ are \emph{$\Sm$-premotivic categories} (i.e., categories also equipped with $f_*$ and $\iHom{},$ see Definition~\ref{defi:Smpremotiviccategory}).

From the generators-and-relations point of view, the top row of the above diagram corresponds to generators, and the second and last row are the relations. The passage from $\Sm / S$ to $\PSh(\Sm / S)$ ``freely'' adds sums of morphisms and colimits of objects to $\Sm / S$ (cf. the fact that every presheaf is a colimit of representable presheaves) and grants the existence of certain adjoint functors. 

The passage to symmetric sequences $\PSh(\Sm / S)^\Sfrak$ is the first step in making the functor $-\otimes M(\G_m / 1)$ invertible. 
The $n$th entry in a symmetric sequence will end up corresponding to the image of the functor $- \otimes M(\G_m / 1)^{\otimes{(-n)}}$, cf. Lemma~\ref{prop:stability}. Keeping track of the $\Sfrak_n$-action makes the tensor product well defined at the underived level.
Having a tensor inverse to $M(\G_m / 1)$ is a necessary condition for duality, and as Ayoub shows, also a sufficient condition in our setting. It is perhaps quite a natural requirement when one remembers that the symbol $M(\G_m / 1)$ represents the reduced cohomology of $\G_m;$ certainly the sheaf of $\ell$th roots of unity $\mu_\ell$ (resp. Tate--Hodge structure $\Z(1)=2\pi i \Z$) should be $\otimes$-invertible in any reasonable category of $\Z / \ell$-{\'e}tale sheaves (resp. mixed Hodge modules). In fact, the \emph{Tate twist} $M(1)$ of a motive $M$ can simply be defined as $M\otimes M(\G_m / 1)[-1]$ and should yield an autoequivalence of our category.

In the second row, we impose Nisnevich descent (or equivalently, impose the Nisnevich--Mayer--Vietoris triangle to be distinguished, or equivalently, impose {\'e}tale excision) by passing from presheaves to Nisnevich sheaves\footnote{In our construction, we actually skip the step of considering $\PSh(\Sm / S)^\Sfrak$, and instead pass directly from $\PSh(\Sm / S)$ to $\Sh_\Nis(\Sm / S)$ and to $\Sh_\Nis(\Sm / S)^\Sfrak.$}. We remark that some authors work with presheaves until the last step of their construction and add this ``relation'' to the Verdier subcategory $\Ss_{\Htp, \Stb}$.

In the last row, we pass to $\K_S\modules.$ This forces certain Hom groups to calculate mod $p$ Milnor $K$-theory and hence ultimately Chow groups as Geisser--Levine show. This grounds our formalism in reality and will make practical computations possible, especially since the Chow and cohomology groups of many spaces considered in geometric representation theory coincide. The transition to $D(\K_S\modules)$ forces quasi-isomorphisms to become isomorphisms, and the final passage to $\H(S)$ imposes $\A^1$-invariance, and forces the ``shift'' operation of symmetric sequences to be isomorphic to $- \otimes M(\G_m / 1),$ with the consequence that this functor and hence the Tate twist becomes invertible.
\subsection{Presheaves and the five operations $f_\#$, $f^*$, $f_*$, $\otimes$, $\iHom{}$.}

In this section we define the categories of sheaves we will use to build $\H$ and recall the language of $\Sm$-premotivic categories used to work with the various functoriality properties which will induce those of $\H$.

\begin{notation}
We set the following notation.
\begin{enumerate}
 \item[{$k$}] is a perfect field of positive characteristic $p$.
 \item[{$\QProj / k$}] is the category of quasi-projective $k$-varieties.
 \item[{$\Sm$}] is the class of smooth morphisms in $\QProj / k$.
 \item[{$\Sm / S$}] is the category of smooth morphisms $X {\to} S$ in $\QProj / k$ for $S$ a quasi-projective variety. Morphisms are commutative triangles $Y {\to} X {\to} S$ in $\QProj / k$ (the morphism $Y {\to} X$ does not have to be smooth). 
 \item[{$\PSh(\Sm / S)$}] is the category of presheaves of abelian groups on $\Sm / S$ where $S \in \QProj / k$.
 \item[{$\Z(-),$}] is the Yoneda embedding (combined with the free abelian presheaf functor)
 \begin{equation}
 \Z(-) : \Sm / S \to \PSh(\Sm / S), \qquad X \mapsto \Z \Hom{}(-, X)
\end{equation}
 \item[{$\Z(X/Y)$}] is the cokernel of the canonical morphism $\Z(Y) \to \Z(X)$ where $Y \to X$ is an immersion in $\Sm / S$ for some $S \in \QProj / k$.
\end{enumerate}
\end{notation}
Notice that the assignment $S \mapsto \Sm / S$ is equipped with functors
\begin{align}
 f^* = T \times_S - &: \Sm / S \to \Sm / T, & \label{equa:Smfstar}\\
 \otimes = \times_S &: \Sm / S \times \Sm / S \to \Sm / S,  &\textrm{ and } \label{equa:Smotimes}\\
 f_\# = \Forget &: \Sm / T \to \Sm / S & (\textrm{when } f:T \to S \textrm{ is smooth}). \label{equa:Smfhash} 
\end{align}

However, there is no internal hom, and $f^*$ does not have a right adjoint. To obtain these, we pass to the categories of presheaves using Yoneda. That is, we consider the assignment $S \mapsto \PSh(\Sm / S)$. It is useful to have a name for the structure we obtain. 

For the convenience of the reader we first quickly recall the notion of a 2-functor.
\begin{definition}[{cf. \cite[\S 2]{Del01}}] \label{defi:2functor}
A 2-functor 
$\Mm: (\QProj / k)^{\op} \to \Cat$ is an assignment sending every variety $S \in \QProj / k$ to a category $\Mm(S)$, every morphism $f: T \to S$ in $\QProj / k$ to a functor $f^*: \Mm(S) \to \Mm(T)$, and every pair of composable morphisms $\stackrel{g}{\to}\stackrel{f}{\to}$ to a natural isomorphism $\alpha_{g,f}: g^*f^* \stackrel{\sim}{\to} (fg)^*$, such that for any triple of composable morphisms $\stackrel{h}{\to}\stackrel{g}{\to}\stackrel{f}{\to}$, the cocycle condition $\alpha_{h, gf} (h^*\alpha_{g,f}) = \alpha_{hg, f} (\alpha_{h,g}f^*)$ is satisfied.
\end{definition}

\begin{definition}[{cf. \cite[Section 1]{CD}}] 
\label{defi:Smpremotiviccategory}
An \emph{$\Sm$-premotivic category} on $\QProj / k$ is a 2-functor $\Mm$, cf. %
\cite[\S 2]{Del01}, factoring through the category of symmetric monoïdal categories, %
satisfying the following properties.

\begin{enumerate}
 \item \textbf{(Adjoints)}
 \begin{enumerate}
  \item \label{defi:Smpremotiviccategory:rightadj} For every morphism $f: T \to S$ the functor $f^*$ has a right adjoint $f_*$.
 
  \item \label{defi:Smpremotiviccategory:leftadj} and when $f$ is in $\Sm$, the functor $f^*$ has a left adjoint $f_\#$.

 If $f = \id$, these are all the identity functors.
 \end{enumerate}

  \item \label{defi:Smpremotiviccategory:baseChange} 
  \textbf{(Base change)} Given a cartesian square 
   \[ \xymatrix{
 Y \ar[r]^q \ar[d]_g & X \ar[d]^f  \\
 T \ar[r]_p & S 
 } \]
with $f$ (and therefore $g$) smooth, the canonical natural transformation $q_\#g^* \stackrel{\sim}{\to} f^* p_\#$ defined in \cite[1.1.6]{CD} is an isomorphism.

 \item \label{defi:Smpremotiviccategory:monoidalComplete}
 \textbf{(Completeness)} Each of the symmetric monoïdal categories $\Mm(S)$ is complete in the sense that 
for every object $F \in \Mm(S)$ the functor $-\otimes F$ admits a right adjoint $\iHom{S}(F, -)$.

 \item \label{defi:Smpremotiviccategory:projForm} 
 \textbf{(Projection formula)} For any smooth $f$, the canonical natural transformation $f_\#(-\otimes f^*-) \stackrel{\sim}{\to} (f_\#-) \otimes (-)$ defined in \cite[1.1.24]{CD}, is an isomorphism, .
\end{enumerate}
\end{definition}

\begin{definition} \label{defi:tripremot}
A \emph{triangulated $\Sm$-premotivic category} is a $\Sm$-premotivic category whose underlying 2-functor comes equipped with a factorisation through the category of symmetric monoïdal triangulated categories.%
\end{definition}

It is observed in \cite[Exam.5.1.1]{CD} that the assignment sending a variety $S$ in $\QProj / k$ to the category $\PSh(\Sm / S)$ and a morphism $f: T \to S$ in $\QProj / k$ to the functor $f^*: \PSh(\Sm / S) \to \PSh(\Sm / T)$ is an $\Sm$-premotivic category (set $\Scal = \QProj / k, \Pcal = \Sm$, and $\Lambda = \Z$ in their notation). The verification of all these properties for $\PSh(\Sm / -)$ is a routine formal exercise.

Notice that the Yoneda embedding preserves $f_\#, f^*, \otimes$. We will also need a name for such a system of natural transformations, and therefore also for an $\Sm$-premotivic category missing $f_*$ and $\iHom{}$.

\begin{definition}[{\cite[Def.1.1.2, Def.1.1.21, Def.1.1.27]{CD}, \cite[Def.1.2.2, Def.1.2.7]{CD}}] \label{definition:monoidalsmfibered}
A \emph{monoïdal $\Sm$-fibered category} is a 2-functor $\Mm$ taking values in symmetric monoïdal categories as in Definition~\ref{defi:Smpremotiviccategory}, satisfying \eqref{defi:Smpremotiviccategory:leftadj}, \eqref{defi:Smpremotiviccategory:baseChange}, %
 \eqref{defi:Smpremotiviccategory:projForm}, but not necessarily \eqref{defi:Smpremotiviccategory:rightadj} or \eqref{defi:Smpremotiviccategory:monoidalComplete}.
A \emph{morphism of monoïdal $\Sm$-fibered categories}, $\phi: \Mm \to \Nn$, is the data of a (strong) symmetric monoïdal functor $\phi_S: \Mm(S) \to \Nn(S)$ for every $S \in \QProj / k$, and natural isomorphisms $f^* \phi_S \cong \phi_T f^*$ of symmetric monoïdal functors for every $f: T \to S$ in $\QProj / k$. These $\phi_S$ are required to satisfy the appropriate cocycle condition with respect to composition in $\QProj / S$, and for smooth morphisms $p: T \to S$, the induced natural transformations $p_\# \phi_T^* \to \phi_S^*p_\#$ (defined in \cite[1.2.1]{CD}) is required to be an isomorphism.
\end{definition}

Thus, the assignments $S \mapsto \Sm/S$ and $S \mapsto \PSh(\Sm / S)$ are mono{\"i}dal $\Sm$-fibered categories, and the Yoneda embeddings
\[ \Z_S: \Sm/S \to \PSh(\Sm/S) \]
define a morphism of mono{\"i}dal $\Sm$-fibered categories; one can check that we have
 \begin{align*}
 f^*\Z_S(X) &\cong \Z_T(T{\times_S}X),\\  \Z_S(X) \otimes \Z_S(X') &\cong \Z_S(X {\times}_S X')  \textrm{ and }\\  f_\#\Z_T(Y) &\cong \Z_S(Y) 
\end{align*}
for $f: T \to S$ a morphism, $X, X' \in \Sm / S$, $Y \in \Sm / T$. The latter isomorphism only makes sense when $f$ is smooth, and we have written $\Z_S, \Z_T$ instead of just the usual $\Z$ to emphasize the base scheme.
 
Of course, every $\Sm$-premotivic category is a monoïdal $\Sm$-fibered category. When we talk about morphisms of $\Sm$-premotivic categories, we mean morphisms as monoïdal $\Sm$-fibered category.
\subsection{Nisnevich sheaves}

In this section we pass from presheaves to Nisnevich sheaves.

\begin{definition}
We equip each $\Sm/S$ with the \emph{Nisnevich topology}. Its covering families are those families of étale morphisms $\{U_i \to X\}_{i \in I}$ which have ``constructible'' sections, by which we mean there exists a sequence of closed subvarieties  $Z_0 \subset Z_1 \subset \dots \subset Z_n = X$ such that for each $j = 1, \dots, n$ there is a factorisation $Z_j {-} Z_{j {-} 1} \to U_{i_j} \to X$ for some $i_j \in I$.
\end{definition}

\begin{remark} \label{rema:emptyCovering}
By convention one also says the empty set is a covering family of the empty variety. This forces every Nisnevich sheaf $F$ to satisfy $F(\varnothing) = 0$.
\end{remark}

Since all étale morphisms are in $\Sm$, this topology is $\Sm$-admissible in the sense of \cite[Def.5.1.3]{CD}.

\begin{notation} \label{nota:NisShv}
We write
\begin{enumerate}
 \item[{$\Sh_{\Nis}(\Sm / S)$}] for the category of Nisnevich sheaves of abelian groups on $\Sm / S$, and

 \item[{$\Z_\Nis(-)$}] when we want to emphasise that the Yoneda embedding $\Z(-)$ takes values in the category of Nisnevich sheaves (the Nisnevich topology is subcanonical, i.e., representable presheaves of sets are Nisnevich sheaves).

 \item[{$\Z_\Nis(X/Y)$}] As above, when $Y \to X$ is an immersion in $\Sm / S$ for some $S \in \QProj / k$, we write $\Z_\Nis(X/Y)$ for the cokernel in the category of Nisnevich sheaves of the canonical morphism $\Z_\Nis(Y) \to \Z_\Nis(X)$. 
 
 \item[{$\Z_\Nis(\G_m/1)$}] $= \coker(\Z_\Nis(S) \stackrel{1}{\to} \Z_\Nis(\G_m))$ where $1: S \to \G_m$ is the section corresponding to $1 \in \OO^*_S$.
\end{enumerate}
\end{notation}

See the beginning of \cite[Sec.3]{MV99} for some motivation for the choice of the Nisnevich topology.

\begin{theorem}[{\cite{MV99}, \cite[cf. proof of 3.3.2]{CD}}]
\label{theo:NBG}
Let $S$ be a quasi-projective variety and $K$ an (unbounded) complex of Nisnevich sheaves on $\Sm / S$. Then the following two conditions are equivalent.
\begin{enumerate}
 \item \label{theo:NBG:1} $H^n(X, K) \to \Hyp_\Nis^n(X, K)$ is an isomorphism for every $X \in \Sm / S, n \in \Z$.
 \item \label{theo:NBG:2} $K(\varnothing)$ is acyclic and $\Cone(K(X) {\to} K(U)) \to \Cone(K(V) {\to} K(U{\times}_XV))$ is a quasi-isomorphism %
  for every cartesian square
  \begin{equation} \label{equa:distNisSqu}
 \xymatrix{ 
 U{\times}_X V \ar[r]^-j \ar[d]_-g & V \ar[d]^f \\
 U \ar[r]_{i} & X
 } \end{equation}
such that $i$ is an open immersion, $f$ is a étale morphism, and $f^{-1}(X{-}U) \to X{-}U$ is an isomorphism. \end{enumerate}
\end{theorem}

It is observed in \cite[Exam.5.1.4]{CD} that the assignment $S \mapsto \Sh_{\Nis}(\Sm/S)$ is also a $\Sm$-premotivic category, so we have all the functors, natural transformations, and isomorphisms mentioned in Definition~\ref{definition:monoidalsmfibered}. Checking this is again a formal routine exercise. The Nisnevich sheaf %
version of the Yoneda embedding gives us another morphism %
of monoïdal $\Sm$-fibered categories.
\begin{equation}
 \Z_\Nis(-): \Sm / S \to \Sh_{\Nis}(\Sm / S). 
\end{equation}

\subsection{Symmetric sequences, the first step towards Stability}

In this section we recall what a symmetric sequence is, cf. \cite[\S 2.1]{HSS}, and highlight various functors, $i_n, s^n, t^n$, which will become important later. This is the first step towards formally $\otimes$-inverting $\Z_\Nis(\G_m / 1) = \coker(\Z_\Nis(S) \stackrel{1}{\to} \Z_\Nis(\G_m))$. %

\begin{remark}
Having a tensor inverse to $M(\G_m / 1)$ is a necessary condition for duality, and as Ayoub shows, also a sufficient condition in our setting. This is perhaps quite a natural requirement when one remembers that the symbol $M(\G_m / 1)$ represents the reduced cohomology of $\G_m$, and certainly $\mu_p$ (resp. $2\pi i \Z$) should be $\otimes$-invertible in any reasonable category of $\Z / p$-{\'e}tale sheaves (resp. mixed Hodge modules).
\end{remark}

\begin{notation} \label{nota:symSeq}
Consider the following categories.
\begin{enumerate}
 \item[{$\Sfrak$}] will denote the category whose objects are finite (or empty) sets, and morphisms are bijections of sets.
 
  \item[{$\Sh_{\Nis}(\Sm / S)^\Sfrak$}] is the category of symmetric sequences in $\Sh_{\Nis}(\Sm / S)$. This is the category of functors from $\Sfrak$ to $\Sh_{\Nis}(\Sm / S)$. Equivalently, it is the category of sequences of sheaves $(\Ee_0, \Ee_1, \dots, )$ such that each $\Ee_n$ is equipped with an action of the symmetric group on $n$ letters $\Sfrak_n$.
 
 \item[{$\sotimes$}] will be the product on $\Sh_{\Nis}(\Sm / S)^\Sfrak$ defined as follows. Given symmetric sequences $\Ee, \Ff$ the symmetric sequence sends a finite set $N$ to 
 \begin{equation}
 (\Ee \sotimes \Ff)(N) = \bigoplus_{N = P \sqcup Q} \Ee(P) \otimes \Ff(Q) 
\end{equation}
where the sum is indexed by decompositions of $N$ into disjoint two finite sets $P$ and $Q$. Given an isomorphism $\phi: N \stackrel{\sim}{\to} N'$, the morphism $(\Ee \sotimes \Ff)(N) \to (\Ee \sotimes \Ff)(N')$ is the obvious one induced by the induced isomorphisms $P \stackrel{\sim}{\to} \phi(P)$ and $Q \stackrel{\sim}{\to} \phi(Q)$. In the alternative description in which we restrict to the sets $\{1, \dots, n\}$, the product is 
\begin{equation}\label{equa:tensorsymmetriccoord}
 (\Ee \sotimes \Ff)_n = \bigoplus_{p = 0, \dots, n} \Ind_{\Sfrak_p {\times} \Sfrak_{n-p}}^{\Sfrak_n} \Ee_p \otimes \Ff_{n-p}. 
\end{equation}

This product is symmetric in the sense that there are canonical functorial isomorphisms $\Ee \sotimes \Ff \cong \Ff \sotimes \Ee$.

\item[{$\sun$}] is the unit for the tensor product $\sotimes$. Explicitly, it is the symmetric sequence which sends all nonempty finite sets to $0$, and the empty set to the constant sheaf. Equivalently, it is the sequence $(\Z_\Nis
, 0, 0, \dots)$.
\end{enumerate}
\end{notation}

It is straightforward to check that the structure of $\Sm$-premotivic category on $\Sh_{\Nis}(\Sm / -)$ induces one on $\Sh_{\Nis}(\Sm / -)^\Sfrak$ equipped with the tensor product $\sotimes$ and unit $\sun$.

\begin{notation} \label{nota:ssBasics}
Let us give some examples of symmetric sequences, and ways of building symmetric sequences. %
\begin{enumerate}
 \item[{$i_nF$}] For any $F \in \Sh_{\Nis}(\Sm / S)$ and $n \geq 0$ consider the sequence which sends a finite set $N$ to 0 if $|N| \neq n$ and $\oplus_{Aut(N)} F$ otherwise. As a sequence this looks like
$
 i_nF = (0, \dots, 0, \oplus_{\Sfrak_n} F, 0, \dots ). 
$
 The functor $i_n$ is equivalently defined as the left adjoint to the functor taking a symmetric sequence to its $n$th space
 \begin{equation}
 i_n: \Sh_{\Nis}(\Sm / S) \rightleftarrows \Sh_{\Nis}(\Sm / S)^\Sfrak : (-)_n. 
\end{equation}

 \item[{$s^n\Ee$}] Let $\Ee\in\Sh_{\Nis}(\Sm / S)^\Sfrak$ be a symmetric sequence. Define $s\Ee$ to be the symmetric sequence which sends a finite set $N$ to $\Ee(N \sqcup \{\ast\})$. For any $n \geq 0$ we set $s^n\Ee = \underbrace{s \circ \dots \circ s}_n\Ee$.
 In the sequence description
$$
 s^n\Ee = (\Ee_n, \Ee_{n+1}, \Ee_{n+2}, \dots) 
$$
where the actions of the $\Sfrak_i$ come from the canonical inclusions $\{1\} \subset \{1, 2\} \subset \{1, 2, 3 \} \subset \dots$.  

 \item[{$t^n\Ee$}]  The functor $s$ has a left adjoint $t$. Setting $t^n = \underbrace{t \circ \dots \circ t}_n$, we have the explicit description 
\begin{equation}\label{equa:defoft}
 t^n\Ee = (\underbrace{0, \dots, 0,}_{n} \Ind_{\Sfrak_0}^{\Sfrak_{n}}\Ee_0, \Ind_{\Sfrak_1}^{\Sfrak_{n + 1}} \Ee_1, \Ind_{\Sfrak_2}^{\Sfrak_{n + 2}} \Ee_2,\dots) 
\end{equation}
for $n \geq 0$. Using the fact that $\Ind_{\Sfrak_i}^{\Sfrak_{n + i}}$ is left adjoint to the forgetful functor from $\Sfrak_{n+i}$ objects to $\Sfrak_i$ objects, one checks easily that 
\begin{equation}
 t^n: \Sh_{\Nis}(\Sm / S)^\Sfrak \rightleftarrows \Sh_{\Nis}(\Sm / S)^\Sfrak : s^n
\end{equation}
is an adjunction for all $n \geq 0$. Moreover, from Equation~\eqref{equa:tensorsymmetriccoord} and \eqref{equa:defoft} we have
\begin{equation}\label{equa:tandtensor}
t\Ee=\Ee\sotimes t\sun.
\end{equation}
\end{enumerate}
\end{notation}
\begin{remark} \label{rema:stexact}
Notice that both $s$ and $t$ are exact functors. 
\end{remark}

\subsection{$\K$-modules, and calculating motivic cohomology}

In this section we introduce the symmetric sequence $\K$ representing motivic cohomology (at least for smooth varieties).

Consider $\G_m$ not as a variety to which we can apply $\Z_\Nis(-)$, but as the sheaf of groups of invertible global sections $\G_m = \OO^* \in \Sh_{\Nis}(\Sm / S)$. This sheaf sends a smooth $S$-scheme $X$ to the units $\Gamma(X, \OO_X)^* = \Gamma(X, \OO_X^*)$ of $\Gamma(X, \OO_X)$, or equivalently, $\Hom{\Sm / S}(X, \G_m)$. Recall that the Milnor $K$-theory $K^M_\bullet(R)$ of a ring $R$ is the quotient of the tensor algebra of the units $(R^*)^{\otimes \bullet}$ by the two-sided ideal generated by elements of the form $a \otimes (1 - a)$ for $a \in R^* {-} \{1\}$, \cite[\S 1]{Mil69}. 

\begin{notation} \label{nota:Kn}
Here we define the Milnor $K$-theory sheaves in $\Shv_\Nis(\Sm/S)$.
\begin{enumerate}
 \item[{$\uK^M_n$}] The sheaf $\uK^M_2$ is defined as the cokernel of the morphism
\begin{equation} \label{equa:St2}
 \textsf{St}_2: \Z(\G_m{-}\{1\}) \to \G_m^{\otimes 2}, \qquad \sum n_ia_i \mapsto \sum n_i (a_i \otimes(1 - a_i)). 
\end{equation}
and more generally, for $n > 1$, the sheaf $\uK^M_n$ is defined as the cokernel of the morphism
\begin{equation} \label{equa:Stn}
\textsf{St}_n \stackrel{def}{=} \sum_{i = 0}^{n - 2}
\id_{\G_m^{\otimes {i}}} \otimes \textsf{St}_2 \otimes \id_{\G_m^{\otimes {n - i}}} :
\bigoplus_{i = 0}^{n - 2}
\G_m^{\otimes {i}} \otimes \Z(\G_m{-}\{1\}) \otimes \G_m^{\otimes {n - i}}
\to \G_m^{\otimes n}.
\end{equation}

 \item[{$\kmp[n]$}] Tensoring with $\Zop$ we obtain the sheaf $\kmp[n] = \uK^M_n \otimes \Zop$.
\end{enumerate}
\end{notation}

\begin{notation} \label{nota:monoidsTK}
We now define two monoïd objects in $\Sh_{\Nis}(\Sm / S)^\Sfrak$ that interest us.
\begin{enumerate}
 \item[{$\T$}] is the commutative monoïd, cf.~Notation~\ref{nota:NisShv},
\[ \T = %
(
\Z_\Nis,\  
\Z_\Nis(\G_m / 1),\ 
\Z_\Nis(\G_m / 1)^{\otimes 2},\ 
\Z_\Nis(\G_m / 1)^{\otimes 3},\ 
\dots, ). \]

 \item[{$\K$}] The obvious action of $\Sfrak_n$ on $\G_m^{\otimes n}$ descends to an action on $\kmp[n]$ and produces a symmetric sequence
  \[ \K = ( 
 \underset{\overset{=}{\Z_\Nis / p}}{\kmp[0]}, 
 \underset{\overset{=}{ \OO^* / (\OO^*)^p}}{\kmp[1]},
  \kmp[2], \ %
  \kmp[3], \dots ). \]
The canonical morphisms $(\OO^*)^{\otimes N} \otimes (\OO^*)^{\otimes M} \to (\OO^*)^{\otimes {N \sqcup M}}$ induce morphisms of symmetric sequences
\begin{align*}
 \T \sotimes \T &\to \T \text{ and }\\
 \K \sotimes \K &\to \K 
\end{align*}
compatible with the symmetry isomorphism of $\sotimes.$ %
We thus obtain a commutative monoïds in $\Sh_{\Nis}(\Sm / S)^\Sfrak$.

In fact, the canonical morphism $\Z_\Nis(\G_m / 1) \to \G_m = \OO^*$ induces a morphism of symmetric sequences 
\begin{equation}
 \T \to \K
\end{equation}
which gives $\K$ the structure of a $\T$-algebra.

 \item[{$\T_S, \K_S$}] If we want to emphasise which base our sheaves are over we will write $\T_S$ and $\K_{S}$.
 
 \item[{$\R\modules.$}] Given a monoïd $\R$, such as $\T$ or $\K$, in a monoïdal category, such as $\Sh_{\Nis}(\Sm / S)^\Sfrak$, we write $\R\modules$ for the category of its modules.

 \item[{$\un, \un_S$}] In a monoïdal category, we denote the tensor unit by $\un$. If we want to emphasize that we consider the unit in $\R_S\modules$, we write $\un_S$.

 \item[{$\rotimes$}] Given a commutative monoïd $\R$, the category $\R\modules$ inherits a canonical structure of symmetric monoïdal category, \cite[Lem.2.2.2]{HSS}. Indeed, for any two $\R$-modules $\Ee, \Ff$ there are two canonical morphisms $\Ee {\sotimes} \R {\sotimes} \Ff \rightrightarrows \Ee {\sotimes} \Ff$ induced by the $\R$-module structures $$\mu_\Ee: \R {\sotimes} \Ee {\to} \Ee\text{ and } \mu_\Ff: \R {\sotimes} \Ff {\to} \Ff$$ of $\Ee\text{ and }\Ff$ respectively, and the symmetry isomorphism $$\sigma_\Ee: \Ee \sotimes \R \to \R \sotimes \Ee$$ of $\sotimes$. We define $\rotimes$ using the coequaliser (i.e., the cokernel of the difference $\id_\Ee \sotimes \mu_\Ff - (\mu_\Ee \sigma_\Ee)\sotimes \id_\Ff$) of these two morphisms
 \begin{equation}
 \Ee \rotimes \Ff = \coker(\Ee \sotimes \R \sotimes \Ff \longrightarrow \Ee \sotimes \Ff). 
\end{equation}

 \item[{$\Ind, \Forget$}] For any commutative monoïd $\R$ in $\Sh_\Nis(\Sm / S)^\Sfrak$ and a symmetric sequence $\Ee$ we obtain a new symmetric sequence just by tensoring with $\R$. This new symmetric sequence has a canonical structure of $\R$-module, induced by the monoïd structure of $\R$. In fact, as one would expect, this process, which we denote $\Ind^\R$ is left adjoint to the forgetful functor $\Forget_\R: \R\modules \to \Sh_\Nis(\Sm / S)^\Sfrak$ which sends an $\R$-module to its underlying symmetric sequence. Similarly, if $\R \to \R'$ is a morphism of commutative monoïds, applying $\R' \rotimes- $ gives a functor $\Ind^{\R'}_\R$ from $\R\modules$ to $\R'\modules$. Just as in the classical case, there is a canonical factorisation of the ``free-module/forgetful-functor'' adjunctions
\begin{equation} \label{equa:freeModuleAdjunctions}
 \xymatrix@C+=1.5cm{
\Sh_\Nis(\Sm / S)^\Sfrak \ar@<0.25em>[r]^-{\Ind^\R} \ar@<1em>@/^1em/[rr]^-{\Ind^{\R'}}  &
\R\modules \ar@<0.25em>[r]^-{\Ind^{\R'}_\R} \ar@<0.25em>[l]^-{\Forget_\R} & 
\R'\modules \ar@<0.25em>[l]^-{\Forget^{\R}_{\R'}} \ar@<1em>@/^1em/[ll]^-{\Forget_{\R'}}
 } 
 \end{equation}
 \item[{$t$}] Note that for any $\R$-module $\Ee$ the symmetric sequence $t(\Forget_\R\Ee)$ comes equipped with a canonical $\R$-module structure, see Equation \eqref{equa:tandtensor}. That is, the functor $t$ has a canonical extension to $\R\modules$ compatible with $\Forget_\R$. Moreover we have
 \begin{equation} \label{equa:tCommutesWithTensor}
 \Ee\rotimes(t\Ff)=t(\Ee\rotimes\Ff)
 \end{equation} for all $\Ee,\Ff\in\R\modules$.
 \item[{$\R(X)$}] Composing all the left adjoints, we find a functor
  \[  \xymatrix{ 
  \R(-) : \Sm / S \ar[r]^-{\Z_\Nis} &
  \Shv_\Nis(\Sm / S) \ar[r]^-{i_0} &
   \Shv_\Nis(\Sm / S)^\Sfrak \ar[r]^-{\Ind^\R} &
    \R\modules } \]

 \item[{$\R(X / Y)$}] If $Y \to X$ is an embedding of smooth $S$-varieties we will write $\R(X / Y) = coker(\R(Y) \to \R(X))$. In particular, we will be using
 \begin{equation}
 \R(\G_m  / 1) = \coker(\R(S) \stackrel{1}{\to} \R(\G_m )). 
\end{equation}

\end{enumerate}
\end{notation}
Not only are $\T$ and $\K$ commutative monoïds, but the two collections $\{ \T_S \}_{S \in \QProj / k}$, and $\{ \K_S \}_{S \in \QProj / k}$ are \emph{cartesian} in the following sense.

\begin{definition}[{\cite[1.1.38, 7.2.10]{CD}}]
A \emph{cartesian section} of a monoïdal $\Sm$-fibered category $\Mm$ is a collection of objects $\{A_X \in \Mm(X)\}_{X \in \QProj / k}$ equipped with isomorphisms $f^*A_X \stackrel{\sim}{\to} A_Y$ for every $f: Y \to X \in \QProj / k.$ These isomorphisms are subject to coherence identities, \cite[Exp.VI]{SGA1}. %

A cartesian section $\R = \{\R_X\}$ of $\Sh_\Nis(\Sm / -)^\Sfrak$ such that each $\R_X$ is equipped with a monoïd structure, and each $f^*\R_X \to \R_Y$ is a morphism of monoïds will be called a \emph{cartesian monoïd}. We define similarly a \emph{cartesian commutative monoïd}.
\end{definition}

\begin{lemma}
For any morphism $f: Y \to X$ in $\QProj / k$ the canonical comparison morphisms $f^*\K_X \to \K_Y$ (described in the proof below) are isomorphisms, and make the collection of $\K_X$ a cartesian section. The same is true for $\T$.
\end{lemma}

\begin{proof}
First we observe that if $W \to X$ is any smooth scheme, and $h_W$ is the Nisnevich sheaf of sets on $\Sm/X$ it represents, then we have $i^*h_W \cong h_{Y {\times_X} W}$ as Nisnevich sheaves on $\Sm/Y$. So in particular, $f^*\G_{m,X} \cong \G_{m,Y}$, and the collection of sheaves $\G_m$ is a cartesian section of $\Sh_\Nis(\Sm / -)$. Moreover, $f^*$ preserves $\otimes$ so the collection of $\G_{m}^{\otimes n}$ is also a cartesian section for any $n \geq 0$. Now recall that $\uK^M_n$ is defined as the cokernel of the morphism $\textsf{St}_n$, Equation~\eqref{equa:Stn}. Observing that $f^*\textsf{St}_{n,X} = \textsf{St}_{n,Y}$ for any $f:Y \to X \in \QProj / k$, and $f^*$ preserves cokernels, we find that the collection of $\uK_{n}^M$ is a cartesian section of $\Sh_\Nis(\Sm / -)$. Finally, $\kmp[n] = \coker(\uK_{n}^M \stackrel{p}{\to} \uK_{n}^M)$ so, again since $f^*$ preserves cokernels, the collection of $\kmp[n]$ is a cartesian section. It follows that $\K$ is a cartesian monoïd.
The statement for $\T$ follows directly from $\Z_\Nis$ being a morphism of monoïdal $\Sm$-fibered categories.
\end{proof}

Cartesiannes is of interest as the assignment $S \mapsto \R_S\modules$ inherits a canonical structure of $\Sm$-premotivic category whenever $\R$ is a cartesian commutative monoïd, \cite[Prop.5.3.1]{CD}. %

\begin{remark} \label{rema:IndRbracketsMormonSmfib}
It is a straightforward exercise to check that the functors $\Ind^\R$ and $\Ind^{\R'}_{\R}$ of Equation~\eqref{equa:freeModuleAdjunctions} are morphisms of monoïdal $\Sm$-fibered categories, and consequently, the functors $\R(-): \Sm / - \to \R_-\modules$ %
form a morphism of monoïdal $\Sm$-fibered categories.
\end{remark}

\begin{remark} \label{rema:FgtExact}
It is also straightforward to check that the functors $\Forget$ are exact in the sense that they preserve all limits and colimits (just as in the case of classical rings).
\end{remark}

\subsection{Derived categories} \label{subsec:DerCat}

In this section, we discuss the derived category $D(\K\modules)$ of $\K$-modules. Our category $\H$ will be the full subcategory of $D(\K\modules)$ of those objects satisfying a stability and homotopy invariance property. We also discuss in this section the interplay between these two properties and the functorialities we have obtained so far.

\begin{notation}
Let $\R_S \in \Sh_\Nis(\Sm / S)^\Sfrak$ be a commutative monoïd such as $\K_S$ or $\T_S$ and consider the following categories.
\begin{enumerate}
 \item[{$C(\Sh_\Nis(\Sm / S))$}] is the category of (unbounded) chain complexes in $\Sh_\Nis(\Sm / S)$.

 \item[{$C(\R_S\modules)$}] is the category of (unbounded) chain complexes in the abelian category $\R_S\modules$. Note that this is equivalent to the category of $\R_S$-modules in $C(\Sh_{\Nis}(\Sm / S))^\Sfrak$, the category of symmetric sequences in the category of unbounded complexes of Nisnevich sheaves of abelian groups.
 
 \item[{$D(\R_S\modules)$}] is the (unbounded) derived category of $\R_S\modules$. 
\end{enumerate}
\end{notation}
 \begin{remark}Note that $D(\R_S\modules)$ is \emph{not} the same as considering $\R_S$-modules in $D(\Sh_{\Nis}(\Sm / S)^\Sfrak)$. In general, for a monoïd $M$ in a tensor triangulated category, it is rare for the category of $M$-modules to inherit the structure of a triangulated category. Similarly, $D(\Sh_{\Nis}(\Sm / S)^\Sfrak)$ is not the same as $D(\Sh_{\Nis}(\Sm / S))^\Sfrak$. There are however canonical functors $D(\Sh_{\Nis}(\Sm / S)^\Sfrak) \to D(\Sh_{\Nis}(\Sm / S))^\Sfrak$ and from $D(\R_S\modules)$ to $\R_S$-modules in $D(\Sh_{\Nis}(\Sm / S)^\Sfrak)$.
 \end{remark}
When $\R$ is cartesian, the systems of categories $C(\R_S\modules)$ and $D(\R_S\modules)$ (as $S$ varies in $\QProj / k$) inherit structures of $\Sm$-premotivic categories. The case $C(\R_S\modules)$ is straightforward \cite[Lemma 5.1.7]{CD}. The case $D(\R_S\modules)$ uses the theory of descent structures developed \cite{CDHomAlg} to observe that the functors $f_\#, f^*, f_*, \otimes, \iHom{}$ of the $\Sm$-premotivic category $C(\R_-\modules)$ can be derived, \cite[5.1.16]{CD}. The model structure is recalled in Section~\ref{subsec:descentModelStru} but we will only need the following consequences.

\begin{proposition} \label{prop:derivedKmodules}
Let $\R$ be a cartesian commutative monoïd of $\Sh_\Nis(\Sm / -)^\Sfrak$, for example, $\R = \K$ or $\R = \T$. %
The functor $\R(-)$ composed with the canonical functor $\R\modules \to C(\R\modules)$ ${\to}$ $D(\R\modules)$ induces a morphism of monoïdal $\Sm$-fibered categories
\begin{equation}
 \R(-): \Sm / - \to D(\R_-\modules). 
\end{equation}

That is, for any morphism $T \stackrel{f}{\to} S \in \QProj/k$, $X, X' \in \Sm / S$, $Y \in \Sm / T$ we have 
\begin{align}
 f_\#\R(Y) &\cong \R(Y) 
 \quad 
 \textrm{ in } D(\R_S\modules)
 \textrm{ (when } f \textrm{ is smooth), } \\
 f^*\R(X) &\cong \R(T{\times_S}X)
 \quad 
 \textrm{ in } D(\R_T\modules)
 , \textrm{ and } \\
 \R(X) \otimes \R(X') &\cong \R(X {\times}_S X')
 \quad 
 \textrm{ in } D(\R_S\modules). \label{prop:derivedKmodules:otimes}
 \end{align}
where the functors $f_\#, f^*, \otimes$ on the left are the derived ones acting on $D(\R\modules)$. Moreover, for any $n \geq 0, X \in \Sm / S$ and $\Ee \in D(\R\modules)$,  
\begin{equation} \label{equa:NisHyperInDKMod}
\Hom{D(\R\modules)}(t^n\R(X), \Ee[i]) \cong \Hyp_\Nis^i(X, \Ee_n). 
\end{equation}
where $\Ee_n$ is the complex obtained by applying the ``$n$th sheaf'' functor $(-)_n$ to the complex of symmetric sequences $\Ee$ (we forget the $\R$-module structure). 
\end{proposition}

\begin{proof}
The first part comes from the observation that the images of representable sheaves are cofibrant. The second part comes from the definition of fibrancy. See Section~\ref{subsec:descentModelStru} for these definitions.
\end{proof}

One of the reasons to use symmetric sequences to invert the Tate twist instead of just a bookkeeping index, is so that our categories admit small sums. This gives us access to the theory of Bousfield localisations/Brown representability à la Neeman, cf. \cite{Nee96}.

\begin{proposition} \label{prop:miscStatementsDmod}
Let $\R \to \R'$ be a morphism of cartesian commutative monoïds of $\Sh_\Nis(\Sm / -)^\Sfrak$, such as $\T \to \K$. %
We have the following.
\begin{enumerate}
 \item \label{prop:miscStatementsDmod:conservativity} The forgetful functors of Equation~\eqref{equa:freeModuleAdjunctions} directly pass to the derived categories without having to be derived. On the derived categories they are conservative.
 
 \item \label{prop:miscStatementsDmod:compGen} The category $D(\R\modules)$ admits all small sums and is compactly generated by the objects $t^n\R(X)$ for $X \in \Sm / S$ and $n \geq 0$.

 \item On the derived categories, all three forgetful functors have left adjoints, and so the adjunctions of Equation~\eqref{equa:freeModuleAdjunctions} induce adjunctions
\begin{equation} \label{equa:derAdjoints}
 \xymatrix@C+=1.5cm{
D(\Sh_\Nis(\Sm / S)^\Sfrak) \ar@<0.25em>[r]^-{L\Ind^\R} \ar@<1em>@/^1em/[rr]^-{L\Ind^{\R'}}  &
D(\R\modules) \ar@<0.25em>[r]^-{L\Ind^{\R'}_\R} \ar@<0.25em>[l]^-{\Forget_\R} & 
D(\R'\modules). \ar@<0.25em>[l]^-{\Forget^{\R}_{\R'}} \ar@<1em>@/^1em/[ll]^-{\Forget_{\R'}}
} 
\end{equation}

\item The functors $L\Ind$ satisfy 
\begin{equation} \label{equa:leftAdjCalc}
 L\Ind^\R(i_n \Z(X)) \cong t^n\R(X), \qquad  L\Ind^{\R'}_\R(t^n\R(X)) \cong t^n\R'(X)
 \end{equation}
for all $S \in \QProj / k$, $X \in \Sm / S$, $n \geq 0$.

\item \label{prop:miscStatementsDmod:Lotimes} The functors $L\Ind$ define morphisms of monoïdal $\Sm$-fibered categories.

 \item The adjunctions $(L\Ind, \Forget)$ also satisfy a projection formula: the canonical comparison natural transformation is an isomorphism:
 \begin{equation} \label{equa:IndFgtProj} 
L\Ind(- \otimes \Forget(-)) \stackrel{\sim}{\to} L\Ind(-) \otimes (-). 
 \end{equation}
\end{enumerate}
\end{proposition}

\begin{proof}
\begin{enumerate}
 \item This follows directly from the fact that a morphism being a weak equivalence or not (resp. an object being acyclic or not) has nothing to do with the $\R$-module structure.

 \item Compactness follows from Equation~\eqref{equa:NisHyperInDKMod}: sheaf cohomology commutes with sums (in $\Sh_\Nis(\Sm / S)$), the Nisnevich topology has finite cohomological dimension \cite[\S 1.2]{KaSa} and so hypercohomology of unbounded complexes also commutes with sums. That it is a generating set (in the sense of %
 \cite[Def.1.7,1.8]{Nee96}) follows from the above conservativity combined with Equation~\eqref{equa:NisHyperInDKMod}. 

 \item Here we use Brown representability \cite[Thm.4.1]{Nee96}: %
 The forgetful functors commute with products, therefore they admit left adjoints.
 
 \item The second one follows from the first via commutativity of the Diagram~\eqref{equa:derAdjoints}. %
 By coYoneda, it suffices to show that the two objects in Equation~\eqref{equa:leftAdjCalc} corepresent the same functor. To this end, we observe that there are isomorphisms
\begin{align*}
\Hom{D(\R\modules)}(L\Ind^\R(i_n \Z(X)), \Ee) 
&\stackrel{\textrm{adjunction}}{\cong} \Hom{D(\Shv_\Nis)^\Sfrak}(i_n\Z(X), \Forget^\R \Ee) \\
\stackrel{}{\cong} \Hyp_\Nis^0(X, \Ee_n) 
&\stackrel{\textrm{Eq.\eqref{equa:NisHyperInDKMod}}}{\cong}
 \Hom{D(\R\modules)}(t^n\R(X), \Ee), 
\end{align*}
 all functorial in $\Ee$.

 \item Let $f: T \to S$ be a morphism in $\QProj / k$. The functors $f_\#$ (when $f$ is smooth), $f^*$, $\otimes$, and $L\Ind$ are all left adjoints, and therefore commute with sums, and our categories are compactly generated, so to show some compatibility relation such as $f^* L\Ind  \cong L\Ind f^*$, it suffices to check it on a set of compact generators. But we have just seen that on our set of compact generators, the functor $L\Ind$ acts as the underived $\Ind$. Moreover, the functors $f_\#$ (when $f$ is smooth), $f^*$, $\otimes$, and $t$ also act on our compact generators as their underived versions; Proposition~\ref{prop:derivedKmodules}, Remark~\ref{rema:stexact}. So $L\Ind$ being a morphism of monoïdal $\Sm$-fibered categories follows from $\Ind$ being a morphism of monoïdal $\Sm$-fibered categories, Remark~\ref{rema:IndRbracketsMormonSmfib}. 

 \item Since all functors in question are triangulated and preserve sums, Remark~\ref{rema:FgtExact}, it suffices to check the morphism on compact generators. We want to show that $L\Ind(t^n\R(X) \otimes \Forget (t^m\R'(Y))) \to L\Ind(t^n\R(X)) \otimes t^m\R'(Y)$ is an isomorphism for all $X, Y \in \Sm / S$ and $n, m \geq 0$. This follows from Equation~\eqref{equa:leftAdjCalc} and the fact that $L\Ind$ preserves $\otimes$.\qed
\end{enumerate}
\renewcommand{\qedsymbol}{}
\end{proof}

\begin{notation} \label{nota:SHtpStb}
Let $\R$ be a cartesian commutative monoïd of $\Sh_\Nis(\Sm / -)^\Sfrak$. %
Consider now the following %
subcategories of $D(\R\modules)$, cf. \cite[5.2.15, 5.3.21]{CD}.
\begin{enumerate}

 \item[{$\Ss_{\Htp}$}] is defined to be the ``$t$-stable'' thick tensor ideal of $D(\R\modules)$ generated by the cone of 
\begin{equation} \label{equa:KA1projection}
 \phi_\Htp: \R(\A^1_S) \to \R(S). 
 \end{equation}
That is, it is the smallest full subcategory of $D(\R\modules)$ containing this cone which is triangulated, closed under direct summands and all small direct sums (thick), and satisfies the properties $\Ss_{\Htp} \otimes \Ee \subseteq \Ss_{\Htp}$ for any $\Ee \in D(\R\modules)$ (ideal), and $t \Ss_{\Htp} \subseteq \Ss_{\Htp}$ ($t$-stable).
\end{enumerate} 

Suppose that $\R$ is equipped with a morphism of monoïds $\T \to \R$.
\begin{enumerate}
 \item[{$\Ss_{\Stb}$}] is the $t$-stable thick tensor ideal of $D(\R\modules)$ generated by the cone of the morphism 
\begin{equation} \label{equa:stabGen}
 \phi_{\Stb}: t \R(\G_m / 1) \to \R(S) 
 \end{equation}
which we will now describe. Consider the adjunction $$\T \sotimes i_1 : \Sh_\Nis(\Sm / S) \rightleftarrows \T_S\modules: (-)_1$$ whose right adjoint sends a $\T$-module to the first (not zeroth) sheaf in the underlying symmetric sequence. %
For $\R = \T$ we define $\phi_\Stb$ to be the counit of the adjunction $(\T \sotimes i_1) \circ (-)_1 (\T) \cong t \T(\G_m / 1) \to \T$, and otherwise we apply $L\Ind_\T^\R$ to the $\phi_\Stb$ for $\T$. It follows from Equation~\eqref{equa:leftAdjCalc} that we get a morphism between the appropriate objects.

 \item[{$\Ss_{\Htp, \Stb}$}] is the smallest $t$-stable thick tensor ideal of $D(\R\modules)$ generated by the cones of the two morphisms \eqref{equa:KA1projection} and \eqref{equa:stabGen}.

 \item[{$\Ss_{?, S}$}] If we want to emphasise the base variety $S$ we will use this notation for $\Ss_{?}$ where $? = $ ``$\Htp$'', or ``$\Stb$'', or ``$\Htp, \Stb$''.

 \item[{$\Ss_{?, S}^{\R}$}] We will use this notation if we want to emphasise the monoïd $\R$.

\end{enumerate}
\end{notation}

\begin{lemma} \label{lemm:altSHtpStbDef}
Let $\R$ be a cartesian commutative monoïd of $\Sh_\Nis(\Sm / -)^\Sfrak$. %
The category $\Ss_{\Htp}$ is equivalently defined as the smallest triangulated subcategory of $D(\R\modules)$ closed under small sums and containing the cones of $\phi_\Htp \otimes t^n\R(X) $ for all $n \geq 0, X \in \Sm / S$. %
If $\R$ is equipped with a morphism of monoïds $\T \to \R$, the analogous statement is true of $\Ss_\Stb$ and $\Ss_{\Htp, \Stb}$.
\end{lemma}

\begin{proof}
This follows directly from the fact that $D(\R\modules)$ is the smallest triangulated subcategory of $D(\R\modules)$ closed under small sums and containing the $t^n\R(X)$ for $n \geq 0, X \in \Sm / S$, %
\cite[Thm.8.3.3]{Nee01}, 
Proposition~\ref{prop:miscStatementsDmod}\eqref{prop:miscStatementsDmod:compGen}. %
\end{proof}

\begin{lemma} \label{fstarPreservesS}
Let $\R$ be a cartesian commutative monoïd of $\Sh_\Nis(\Sm / -)^\Sfrak$ equipped with a morphism of monoïds $\T \to \R$. %
For any morphism of varieties $f : T \to S$ in $\QProj / k$, we have $f^* \Ss_{\Htp, \Stb, S} \subseteq \Ss_{\Htp, \Stb, T}$. If $f$ is smooth we have $f_\# \Ss_{\Htp, \Stb, T} \subseteq \Ss_{\Htp, \Stb, S}$.
\end{lemma}

\begin{proof}
Since $f^*$ preserves $\otimes$, by the definition of $\Ss_{\Htp, \Stb}$ as an ideal it suffices to show that \begin{equation} \label{equa:fstarphi}
f^*(\phi_{?, S}) = \phi_{?, T}, \qquad \textrm{ for }? = \Htp, \Stb. 
\end{equation}
For $\Htp$ this follows from $\R(-)$ commuting with $f^*$, Proposition~\ref{prop:derivedKmodules}. The case of $\Stb$ also follows from this, together with $f^*$ commuting with the adjunction $(\T\sotimes i_1, (-)_1)$ and the functor $L\Ind^\R_\T$, it is just a little fiddly.

For $f_\#$, by Lemma~\ref{lemm:altSHtpStbDef}, we should show that $f_\#$ sends $\phi_{?, T} \otimes t^n\R(X)$ inside $\Ss_{?, S}$ for $? = \Htp, \Stb$. This follows from the Projection Formula, Definition~\ref{defi:Smpremotiviccategory}\eqref{defi:Smpremotiviccategory:projForm}, and Equation~\eqref{equa:fstarphi}.
\end{proof}

\begin{lemma} \label{lemm:IndPreservesSs}
Let $\R$ and $\R'$ be cartesian commutative monoïds in $\Sh_\Nis(\Sm / -)^\Sfrak$ equipped with morphisms of monoïds $\T \to \R \to \R'$. %
For any $S \in \QProj / k$ we have
\begin{equation}  \label{equa:IndPreservesSs}
L\Ind^{\R'}_{\R}(\Ss_{\Htp, \Stb}^\R) \subseteq \Ss_{\Htp, \Stb}^{\R'} \qquad \textrm { and } \qquad \Forget^\R_{\R'} (\Ss_{\Htp, \Stb}^{\R'}) \subseteq \Ss^\R_{\Htp, \Stb}.
\end{equation}
\end{lemma}

\begin{proof}
For the first claim, since $L\Ind^{\R'}_{\R}$ preserves $\otimes$, Proposition~\ref{prop:miscStatementsDmod}\eqref{prop:miscStatementsDmod:Lotimes}, and the $\Ss_{\Htp, \Stb}$ are defined as ideals, it suffices to show that $L\Ind^{\R'}_{\R}$ sends $\phi^\R_\Htp$ and $\phi^\R_\Stb$ inside $\Ss_{\Htp, \Stb}^{\R'}$. One checks that, in fact, we have
\begin{equation} \label{equa:Indphi}
L\Ind^{\R'}_{\R}(\phi^\R_?) \cong \phi_?^{\R'}, \qquad ? = \Htp, \Stb.
\end{equation}

For the second claim we use the alternative definition of Lemma~\ref{lemm:altSHtpStbDef}. We must show that $\Forget^\R_{\R'} (\phi_?^{\R'} \otimes t^n\R'(X))$ is in $\Ss^\R_{\Htp, \Stb}$ for $? = \Htp, \Stb$. This follows from Equation~\eqref{equa:Indphi} and Equation~\eqref{equa:IndFgtProj}.
\end{proof}

\begin{lemma} \label{lemm:Sstblocality}
Let $\R$ be a cartesian commutative monoïd of $\Sh_\Nis(\Sm / -)^\Sfrak$ equipped with a morphism of monoïds $\T \to \R$. 
An object $\Ee \in D(\R\modules)$ is $\Ss_{\Stb}$-local (see Rappels~\ref{rappels:tricat2}\eqref{BousfieldLocalisationTriCat}) if and only if for every $X \in \Sm / S$, $n, i \geq 0$ the morphisms
\begin{equation} \label{equation:Wlocalcriteria}
\Hyp_\Nis^i(X, \Ee_n) \to %
\ker \biggl ( \Hyp_\Nis^i(\G_m {\times_S} X, \Ee_{n+1}) \to \Hyp_\Nis^i(X, \Ee_{n+1}) \biggr )
\end{equation}
corresponding to 
\begin{equation}
 \Hom{D(\R\modules)}((\phi_{\Stb}) \otimes t^n\R(X) , \Ee) 
\end{equation}
under Equation~\eqref{equa:NisHyperInDKMod} are all isomorphisms, where $\phi_{\Stb}$ is the morphism from Equation~\eqref{equa:stabGen}. Similarly, an object is $\Ss_{\Htp}$-local if and only if for all $X \in \Sm / S, n \geq 0, i\in \Z$, the canonical morphisms
\begin{equation}
\Hyp_\Nis^i(X, \Ee_{n})  \to \Hyp_\Nis^i(\A^1_X, \Ee_{n}) 
\end{equation}
are isomorphisms, and an object is $\Ss_{\Htp, \Stb}$-local if and only if it is both $\Ss_{\Htp}$-local and $\Ss_{\Stb}$-local.
\end{lemma}

\begin{proof}
This follows directly from the description of Lemma~\ref{lemm:altSHtpStbDef} and Proposition~\ref{prop:derivedKmodules}\eqref{prop:derivedKmodules:otimes}.
\end{proof}

\subsection{Motives---definitions}

In this section we complete our construction of $\H$.

\begin{setting} \label{sett}
Let $\R$ be a cartesian commutative monoïd of $\Sh_\Nis(\Sm / -)^\Sfrak$ equipped with a morphism of monoïds $\T \to \R$, for example, $\R = \K$ or $\R = \T$. %
\end{setting}

\begin{definition} \label{defi:H}
Let $\R$ be as in Setting~\ref{sett}. 
For $S \in \QProj / k$ define $\H(S, \R)$ as the Verdier quotient
\begin{equation} \label{equa:HSRdef}
 \H(S, \R) \stackrel{def}{=} D(\R_S\modules) / \Ss_{\Htp, \Stb}.
\end{equation}
See Section~\ref{sec:triCat} for some recollections about Verdier quotients.

If $A$ is a commutative $\Zop$-algebra we will write (cf. Notation~\ref{nota:monoidsTK})
\begin{equation} \label{}
 \H(S, A) \stackrel{def}{=} \H(S, A \otimes \K).
\end{equation}
We define the composition $\Sm / S \to D(\R_S\modules) \to \H(S, \R)$ of the canonical functors as 
\[ M(-): \Sm / S \to \H(S, \R). \]
\end{definition}

\begin{remark}
By the definition of $\Ss_{\Htp,\Stb}$, the functor $t$ on $D(\R_S\modules)$ induces a functor on $\H(S, \R)$, which we continue to denote by $t$.
\end{remark}

\begin{remark} \label{rema:DA1def}
In the case $\R = \T$, Cisinski--Déglise use the notation
\begin{equation}
 D_{\A^1}(S) \stackrel{def}{=} \H(S, \T),
\end{equation}
cf. \cite[5.3.21, Def.5.3.22, Exam.5.3.31, Rem.5.3.34]{CD}.
\end{remark}

\begin{proposition} \label{prop:Hispremotivic}
Let $\R$ be as in Setting~\ref{sett}. 
The system of categories $\H(-, \R)$ inherits a structure of triangulated $\Sm$-premotivic category from $D(\R_S\modules)$, cf. Definition~\ref{defi:tripremot}. The canonical functors %
$
 D(\R_S\modules) \to \H(S, \R) 
$
form a morphism of $\Sm$-premotivic categories. Moreover, $\H(S, \R)$ is compactly generated by the $t^nM(X)$ for $n \geq 0, X \in \Sm / S$.
\end{proposition}

\begin{remark}
In particular, this proposition means that the morphisms in the composition %
$
M(-): \Sm / S \to D(\R_S\modules) \to \H(S, \R) 
$
commute (up to canonical isomorphism) with the functors $f^*$, $\otimes$, and for smooth morphisms, the functors $f_\#$, cf. the end of Section~2.2. So for any $f: S' \to S$ in $\QProj / k$, $X, Y \in \Sm /S$, and $X' \in \Sm / S'$, we have
\begin{align}
f^*M(X) &\cong M(S' {\times}_S X), \\
M(X) \otimes M(Y) &\cong M(X{\times}_S Y), \qquad \textrm{ and } \\
f_\#M(X') &\cong M(X') \qquad (\textrm{when } f \textrm{ is smooth}).
\end{align}
\end{remark}

\begin{proof}
Let $f: T \to S$ be a morphism in $\QProj / k$. By Lemma~\ref{fstarPreservesS}, and the description of $\Ss_{\Htp, \Stb}$ as an ideal, the functors $f_\#$ (if $f$ is smooth), $f^*, \otimes$ preserve $\Ss_{\Htp, \Stb}$. Hence, they descend to the Verdier quotient by the universality in the definition. The adjunctons %
The statement about being compactly generated follows directly from the fact that $D(\R\modules)$ is compactly generated by these objects, and the adjunction of Rappels~\ref{rappels:tricat2}\eqref{BousfieldRightadjoint}.
\end{proof}

\begin{proposition} \label{prop:PhiPsi}
Let $\R, \R'$ be as in Setting~\ref{sett}, 
equipped with morphisms of monoïds $\T \to \R \to \R'$. %
The adjunction $(L\Ind^{\R'}_\R, \Forget^\R_{\R'})$ of Equation~\eqref{equa:derAdjoints} passes to the Verdier quotients:
\begin{equation}
\H(S, \R) \rightleftarrows \H(S, \R') 
\end{equation}
\begin{enumerate}
 \item The left adjoint is a morphism of $\Sm$-premotivic categories.
 \item The right adjoint commutes with $f^*$ for any morphism $f \in \QProj / k$ (but it is not a morphism of monoïdal $\Sm$-fibered categories).
 \item For every $S \in \QProj / k$, the right adjoint is conservative.
\end{enumerate}
\end{proposition}

\begin{proof}
All claims except the last one follow from the universal property of Verdier localisations, %
since the subcategories $\Ss_{\Htp, \Stb}$ are preserved by the functors in question. The last claim follows from the description of $\H(S, \R)$ and $\H(S, \R')$ as full subcategories of $D(\R\modules)$ and
$D(\R'\modules)$ (see Rappels~\ref{rappels:tricat2}\eqref{BousfieldLocalisationTriCat}); the functor $\Forget^\R_{\R'}$ preserves these full subcategories because its left adjoint $L\Ind^{\R'}_{\R}$ preserves the $\Ss_{\Htp, \Stb}$.
\end{proof}

\subsection{Motives---six operations}

In this section we develop the properties of $\H(-)$ that we are interested in.

\begin{proposition} \label{prop:adjointProperty}
Let $\R$ be as in Setting~\ref{sett}. 
For \emph{any} morphism $f: T \to S \in \QProj / k$, the functor $f_*: \H(T, \R) \to \H(S, \R)$ admits a right adjoint.
\end{proposition}

\begin{proof}
By Brown representability, \cite[Thm.4.1]{Nee96}, 
it suffices to show that $f_*$ preserves small sums. Since we are working with compactly generated categories, Proposition~\ref{prop:Hispremotivic}, the right adjoint $f_*$ preserves small sums if and only if its left adjoint $f^*$ preserves compact objects, \cite[Thm.5.1]{Nee96}. In fact, it suffices that $f^*$ sends each compact generator $t^nM(X), X \in \Sm / S, n \geq 0$ to a compact object. But since $M$ and $t^n$ are morphisms of monoïdal $\Sm$-fibered categories, we have $f^*t^nM(X) \cong t^nM(T{\times_S}X)$. 
\end{proof}

\begin{proposition}[Homotopy invariance] \label{prop:homotopy}
Let $\R$ be as in Setting~\ref{sett}. 
The $\Sm$-premotivic category $\H(-, \R)$ is $\A^1$-homotopy invariant in the sense that for any $S \in \QProj / k$ the unit of adjunction $\id \to p_*p^*$ is an isomorphism where $p: \A^1_S \to S$ is the canonical projection.
\end{proposition}

\begin{proof}
By adjunction it suffices to prove that the counit $p_\#p^* \to \id$ is an isomorphism. %
Both $p_\#$ and $p^*$ are left adjoints, and consequently commute with all small sums. So since the category $\H(S)$ is compactly generated by the $t^nM(X)$ for $n \geq 0$ and $X \in \Sm / S$ (see proof of Proposition~\ref{prop:Hispremotivic}) it suffices to show that the morphisms $p_\#p^*t^nM(X) \to t^nM(X)$ are all isomorphisms. But these are isomorphic to the images of the morphisms $\phi_{\Htp} \otimes t^n\R(X)$ by Proposition~\ref{prop:derivedKmodules} and these are isomorphisms by definition since they are used to define $\Ss_{\Htp}$, Lemma~\ref{lemm:altSHtpStbDef}.
\end{proof}

\begin{proposition}[Localisation] \label{prop:localisation}
Let $\R$ be as in Setting~\ref{sett}. 
For any closed immersion $i: Z \to X$ in $\QProj / k$ with open complement $j: U \to X$ the pair 
\begin{equation}
 (i^*, j^*): \H(X,\R) \to \H(Z,\R) \times \H(U,\R) 
\end{equation}
is conservative, and the counit $i^*i_* \to \id$ is an isomorphism.
\end{proposition}

\begin{remark}
Note that localisation for integral motives as defined by Voevodsky is a major open problem. However, since transfers don't appear anywhere in this paper, we can apply what is in some sense \cite[Theorem 3.2.21]{MV99}. The new problem then becomes to show that whatever categories you are working with are equivalent to Voevodsky's categories, but we don't care about this. We only care about the calculation of Corollary~\ref{coro:fieldChowGroups}, and that we have a six functor formalism (including localisation).
\end{remark}

\begin{proof}
Rather than reproduce the proof of \cite[Theorem 3.2.21]{MV99} as most writers do, we will deduce localisation for $\H(-, \R)$ from localisation for $D_{\A^1} = \H(-, \T)$.
Indeed, $D_{\A^1}$ satisfies Localisation, \cite[Theorem 6.2.1]{CD}. Since the right adjoint from Proposition~\ref{prop:PhiPsi} is conservative and commutes with $j^*$, $i^*$, and $i_*$, it follows that $\H(-, \R)$ also satisfies Localisation.
\end{proof}
 
\begin{corollary}
With notation as in Proposition~\ref{prop:localisation} the unit and counit of adjunction fit into a distinguished triangle
\begin{equation} \label{equa:locaTri}
j_\#j^* \to \id \to i_*i^* \to j_\#j^*[1]
\end{equation}
\end{corollary}
 
\begin{proof}
First notice that we have the identities
\begin{equation} \label{equa:locaIden}
i^*i_*i^* \cong i^*, \quad i^*j_\# \cong 0, \quad  j^*j_\# \cong \id, \quad  j^*i_* \cong 0.
\end{equation}
The first one is part of Localisation. The second two follow from the observation that $\H(\varnothing)$ is the zero category and the Base Change property, Definition~\ref{defi:Smpremotiviccategory}\eqref{defi:Smpremotiviccategory:baseChange}. The last one follows from the fact that $j^*i_*$ is the right adjoint to $i^*j_\#$. For any object $\Ee$, choose a cone $\Ff$ of $j_\#j^*\Ee \to \Ee$. It follows from adjunction and $j^*i_* \cong 0$ that $\Hom{}(j_\#j^*\Ee, i_*i^*\Ee) = 0$ and so by the usual triangulated category  long exact $\Hom{}$ sequence, there exists a factorisation $j_\#j^*\Ee \to \Ee \stackrel{}{\to} \Ff \stackrel{\phi}{\to} i_*i^*\Ee$. Applying $i^*$ to these morphisms and using the identities \eqref{equa:locaIden} we obtain the distinguished triangle $0 \to i^*\Ee \stackrel{\sim}{\to} i^*\Ff \to 0[1]$, and the factorisation $i^*\Ee \to i^*\Ff \stackrel{i^*\phi}{\to} i^* \Ee$ which shows that $i^*\phi$ is an isomorphism. Applying $j^*$ to these morphisms and using the identities \eqref{equa:locaIden} we obtain the distinguished triangle $j^*\Ee \to j^*\Ee \to j^*\Ff \to j^*\Ee[1]$ showing that $j^*\Ff$ is zero. Since $j^*i_*i^*\Ee$ is also zero, $j^*\phi: j^*\Ff \to j^*i_*i^*\Ee$ is an isomorphism. Hence, since $(j^*, i^*)$ is conservative, it follows that $\phi$ is an isomorphism. Moreover, applying $\Hom{}(j_\#j^*\Ee[1], i_*i^*\Ee) = 0$, we see that $\Ff$ was unique up to unique isomorphism.
\end{proof}

\begin{lemma} \label{lemm:TateTwist}
Let $\R$ be as in Setting~\ref{sett}. 
There is a natural isomorphism of endofunctors
\begin{equation}
p_\#s_*(-) \cong - \otimes M(\G_m / 1)[1] : \H(S, \R) \to \H(S, \R)
\end{equation}
where $p: \A^1_S \to S$ is the canonical projection and $s: S \to \A^1_S$ the zero section. Consequently, there is also a natural isomorphism of their right adjoints
\begin{equation}
s^!p^*(-) \cong \iHom{\H(S, \R)}(M(\G_m / 1)[1], -)
\end{equation}
where $s^!$ is the right adjoint to $s_*$ given by Proposition~\ref{prop:adjointProperty}, cf. Theorem~\ref{theo:propertiesofH}\eqref{theo:propertiesofH:properRightAdjoint}.
\end{lemma}

\begin{proof}
Consider the localisation triangles, Equation~\eqref{equa:locaTri}, associated to the open immersion $j: \G_m \to \A^1$. This gives rise to distinguished triangles
\begin{equation} \label{equa:A1Gmlocal}
 p_\#(j_\#j^*)p^* \to p_\#p^* \to p_\#(s_*s^*)p^* \to. 
\end{equation}
We have $p_\#p^* \cong \id$ by homotopy invariance, Proposition~\ref{prop:homotopy}, and $p_\#(j_\#j^*)p^*(-) \cong - \otimes M(\G_m) $ by %
the Projection Formula, Definition~\ref{defi:Smpremotiviccategory}\eqref{defi:Smpremotiviccategory:projForm}, 
so our distinguished triangle becomes
\begin{equation}
  (- \otimes M(\G_m)) \to (- \otimes M(S)) \to p_\#s_* \to, 
\end{equation}
The structural morphism $M(\G_m) \to M(S)$ is split by the identity section, and so $M(\G_m)$ decomposes as a direct sum $M(\G_m) \cong M(\G_m / 1) \oplus M(S)$, giving an isomorphism $M(\G_m / 1)[1] \cong \Cone(M(\G_m) \to M(S))$. The result follows.
\end{proof}

\begin{proposition}[Stability] \label{prop:stability}
Let $\R$ be as in Setting~\ref{sett}. 
The $\Sm$-premotivic category $\H(-, \R)$ is stable in the following sense. For any $S \in \QProj / k$ the endofunctor $s^!p^*$ for $p: \A^1_S \to S$, $s: S \to \A^1_S$ as above is an equivalence.
\end{proposition}

\begin{proof}
By Lemma~\ref{lemm:TateTwist} it suffices to show that $\iHom{\H(S, \R)}(M(\G_m / 1), -)$ is an invertible endofunctor. We do this in Lemma~\ref{lemma:stability}.
\end{proof}

\begin{lemma} \label{lemma:stability}
The functor $\iHom{\H(S)}(M(\G_m / 1), -)$ on $\H(S, \R)$ is an equivalence, inverse to the functor induced by $s: (\Ee_0, \Ee_1, \Ee_2, \dots) \mapsto (\Ee_1, \Ee_2, \Ee_3, \dots)$. %
Consequently, there are natural isomorphisms of endofunctors
\begin{equation}
t \cong \iHom{\H(S)}(M(\G_m / 1), -), \qquad s \cong - \otimes M(\G_m / 1) .
\end{equation}
\end{lemma}

\begin{proof}
To evaluate these two functors on $\H(S, \R)$, we consider the full subcategory of $D(\R\modules)$ consisting of $\Ss_{\Htp, \Stb}$-local objects (see Rappels~\ref{rappels:tricat2}\eqref{BousfieldLocalisationTriCat}) which are fibrant, Definition~\ref{defi:fibrantKmodule}. These properties imply that for any such object and any $X \in \Sm / S$ the canonical morphism
\begin{equation}
 H^i(X, \Ee_n) \to H^i(X, \iHom{C(\Sh_\Nis(\Sm / S))}(\Z_\Nis(\G_m  / 1), \Ee_{n + 1})) 
\end{equation}
is an isomorphism, Lemma~\ref{lemm:Sstblocality}. That is, the canonical morphism
\begin{equation}
 \Ee_n \to \iHom{C(\Sh_\Nis(\Sm / S))}(\Z_\Nis(\G_m / 1), \Ee_{n + 1}) 
\end{equation}
is a quasi-isomorphism or in other words the morphism
\begin{equation}
 \Ee \to \iHom{C(\R\modules)}(\R(\G_m / 1), s \Ee) = s \iHom{C(\R\modules)}(\R(\G_m / 1), \Ee) 
\end{equation}
is a quasi-isomorphism, and therefore the corresponding morphism in $D(\R\modules)$ is an isomorphism. So we have shown that we have natural isomorphisms of endofunctors of $H(S, \R)$
\begin{equation}
 \id \cong s \circ \iHom{H(S, \R)}(M(\G_m / 1), -) \cong \iHom{H(S, \R)}(M(\G_m / 1), -) \circ s 
\end{equation}
from which it follows that both $s$ and $\iHom{H(S, \R)}(M(\G_m / 1), -)$ are essentially surjective and fully faithful, and inverse equivalences of categories. The ``Consequently'' statement follows directly by uniqueness of adjoints.
\end{proof}

\begin{notation} \label{nota:tateTwist}
Recall that for $n \in \Z$ we have defined the $n$th \emph{Tate twist} by
\begin{enumerate}
 \item[{$(-)(n):$}] $\H(S, \R)  \to \H(S, \R)$ to be $(s^!p^*)^n(-)[-2n]$
\end{enumerate}
where $p: \A^1_S \to S$, $s: S \to \A^1_S$ as above (see also Theorem~\ref{theo:propertiesofH}~\eqref{theo:propertiesofH:except}, \eqref{theo:propertiesofH:properRightAdjoint} and \eqref{theo:propertiesofH:stability}). By Lemmata~\ref{lemm:TateTwist} and~\ref{lemma:stability} we see that we have when $n \geq 0$ we have 
\begin{equation}
 (-)(-n)[-n] \cong \iHom{}(M(\G_m/S)^{\otimes n}, -) \cong t^n
\end{equation}
and 
\begin{equation}
 (-)(n)[n] \cong - \otimes M(\G_m/S)^{\otimes n}  \cong s^n.
\end{equation}
\end{notation}

\begin{corollary} \label{coro:stableHomotopy2functor}
Let $\R$ be as in Setting~\ref{sett}. 
The 2-functor $\H(-, \R)$ is a \emph{triangulated $\Sm$-motivic category} in the sense of \cite[Def.2.4.45]{CD}, and a unitary symmetric monoïdal stable homotopy 2-functor in the sense of \cite[Def.1.4.1,Def.2.3.1]{Ayo07}. Consequently, we have all the six operations and properties as described in Theorem~\ref{theo:propertiesofH}. See the statement of the theorem for references.
\end{corollary}

\begin{proof}
By definition, a triangulated $\Sm$-motivic category is a triangulated $\Sm$-premotivic category which satisfies Homotopy, Stability, Localisation and the Adjoint Property, the latter being: $f_*$ has a right adjoint for every proper morphism $f$. We have seen all of this in Propositions~\ref{prop:Hispremotivic}, \ref{prop:homotopy}, \ref{prop:stability}, \ref{prop:localisation}, and \ref{prop:adjointProperty}, respectively.

On the other hand, by definition, a stable homotopy 2-functor is a 2-functor $\H$ satisfying: $\H(\varnothing) = 0$, Properties \eqref{defi:Smpremotiviccategory:rightadj}, \eqref{defi:Smpremotiviccategory:leftadj}, \eqref{defi:Smpremotiviccategory:baseChange} of Def.~\ref{defi:Smpremotiviccategory}, Homotopy, Stability, and Localisation. A stable homotopy 2-functor is unitary symmetric monoïdal if it comes equipped with a factorisation through the category of (unital) symmetric monoïdal triangulated categories, %
and satisfies Property \eqref{defi:Smpremotiviccategory:projForm} of Def.~\ref{defi:Smpremotiviccategory}.
\end{proof}

\begin{remark} \label{rema:immersion} Ayoub actually asks that $i^*i_* \to \id$ be an isomorphism for \emph{any} immersion. Since any quasi-projective immersion factors as a closed and open immersion, it suffices to consider the two cases $i$ is a closed immersion and $i$ is an open immersion. The closed immersion case is part of the Localisation property as we state it. For the case when $i$ is an open immersion, notice that Property (4) of Def.~\ref{defi:Smpremotiviccategory} implies that $\id \to i^*i_\#$ is an isomorphism. Then we notice that $i^*i_*$ is right adjoint to $i^*i_\#$, and a left adjoint is an equivalence if and only if its right adjoint is an equivalence.
\end{remark}

\subsection{Motives---Chow groups} \label{sec:motivesChow}

\begin{notation} \label{nota:Chow}
We write
\begin{enumerate}
 \item[{$\Z(n)$}] for the complex denoted by $z^n(X, 2n-\ast)$ in \cite{Blo86}. It is the following complex of presheaves concentrated in cohomological degrees $\leq 2n$. For $X \in \QProj / k$ the group $\Z(n)(X)^i$ is  the free abelian group of codimension $n$ closed irreducible subsets of 
 \[ X \times \Spec(k[t_0, \dots, t_{2n-i}] / 1 - \Sigma t_j) \] whose intersection with each of the faces $t_j = 0$ has pure codimension $n$ inside that face.  The differentials are the alternating sums of the intersections with the faces.
 
 \item[{$A(n)$}] $ = \Z(n) \otimes A$ for any abelian group $A$.

 \item[{$\Chow^n(X, 2n{-}i)$}] is the $i$th cohomology group of $\Z(n)(X)$.
 
 \item[{$\Chow^n(X, 2n{-}i; A)$}] is the $i$th cohomology group of $A(n)(X)$.
\end{enumerate}
\end{notation}

One of the most important properties of the complexes $\Z(n)$ is the localisation property.

\begin{theorem}[{\cite[p.269]{Blo86}}]
For any closed immersion $Z \to Y$ of codimension $d$ in $\QProj / k$ with open complement $U \to Y$, there exists a canonical quasi-isomorphism of complexes of abelian groups 
\begin{equation} \label{equa:Znlocal}
\Z(n{-}d)(Z) \stackrel{q.i.}{\to} \Cone \biggl (  \Z(n)(Y) {\to}  \Z(n)(U)\biggr )[1].
\end{equation}
\end{theorem}

\begin{corollary} \label{coro:ZnNis}
For any $X \in \QProj / k$, $n, i \in \Z$, and any abelian group $A$ we have 
\begin{equation}
H^i(A(n)(X)) \cong \Hyp_\Nis^i(X, A(n))
\end{equation}
\end{corollary}

\begin{proof}
The localisation sequence implies that the presheaf $\Z(n)$ satisfies Condition~\eqref{theo:NBG:2} of Theorem~\ref{theo:NBG}. Since $\Z(n)$ is a presheaf of complexes of free abelian groups, we can apply $-\otimes A$ directly, and $A(n)$ also satisfies Condition~\eqref{theo:NBG:2} of Theorem~\ref{theo:NBG}. Therefore, $A(n)$ also satisfies Condition~\eqref{theo:NBG:1}.
\end{proof}

We will use the following result of Geisser--Levine.

\begin{theorem}[{\cite[Proposition 3.1, Theorem 8.5]{GL}}] \label{theorem:GL}
For any $X \in \Sm / k$ there are canonical functorial isomorphisms
\begin{equation}
 \Chow^n(X, 2n{-}i; \Zop) \cong \Hyp_\Nis^{i-n}(X, \kmp[n]). 
\end{equation}
\end{theorem}

\begin{proof}
In \cite[Theorem 8.5]{GL} Geisser--Levine show that on the small Zariski site of a smooth $k$-variety $X$, the complex $\Zop(n)$ is quasi-isomorphic to the complex concentrated in degree $n$ with the Zariski sheaf $\nu^n$ in degree $n$ (their statement is a little strange, but $\tau_{\leq n}R\epsilon_*\nu^n[-n] = \nu^n[-n]$ due to $\nu^n$ being an étale sheaf). We don't recall what $\nu^n$ is because it doesn't matter to us at the moment, we will just observe that Geisser--Levine's \cite[Proposition 3.1]{GL} says that it is isomorphic to the sheafification of $\kmp[n]$ (since $\nu^n$ is an étale sheaf, \cite[Proposition 3.1]{GL} shows that the Zariski, Nisnevich, and étale sheafifications of $\kmp[n]$ are all the same).

So we have a quasi-isomorphism 
\begin{equation}
 \Zop(n)_\Nis \cong \kmp[n][-n] 
\end{equation}
in $C(\Shv_\Nis(\Sm / k))$. It remains to observe that the cohomology of $\Zop(n)$ is the same as its Nisnevich hypercohomology. This is Corollary~\ref{coro:ZnNis}.
\end{proof}

\begin{corollary} \label{coro:calculateHigherChowGroups}
For any smooth $k$-variety $f: X \to k$, and any commutative $\Zop$-algebra $A$ we have canonical functorial isomorphisms
\begin{equation} \label{equa:ChhomH}
 \Chow^{n}(X, 2n{-}i; A) \cong \Hom{\H(k, \K \otimes A)}(\un, f_*f^*\un(n)[i]). 
\end{equation}
\end{corollary}
\begin{proof}
Assume for the moment that $\K$ is $\Ss_{\Htp, \Stb}$-local. Then we have isomorphisms, \cite[Lemma 9.1.5]{Nee01}
\begin{align}
&\vphantom{\cong}\Hom{\H(k, \K)}(\un, f_*f^*\un(n)[i]) \\
(\textrm{Adjunction}) \qquad
&\cong\Hom{\H(X, \K)}(\un, \un(n)[i]) \\
(\textrm{Nota.\ref{nota:tateTwist}}) \qquad
&\stackrel{}{\cong} \Hom{\H(X, \K)}(t^n\K, \K[i{-}n]) \\
(\Ss_{\Htp, \Stb}\textrm{-locality, Rap.~\ref{rappels:tricat2}\eqref{BousfieldLocalisationTriCat}}) \qquad
&\stackrel{}{\cong} \Hom{D(\K_X\modules)}(t^n\K, \K[i{-}n]) \\
(\textrm{Equa.\eqref{equa:NisHyperInDKMod}}) \qquad 
&\cong \Hyp_{\Nis}^{i-n}(X, \kmp[n]) \\
(\textrm{Thm.\ref{theorem:GL}}) \qquad
&\cong \Chow^n(X, 2n-i; \Zop).
\end{align}
Now $\Ss_{\Htp}$-locality of $\K$ follows from $\A^1$-invariance for $CH$, \cite[p.269]{Blo86}. For $\Ss_{\Stb}$-locality, use the Projective Bundle Theorem, \cite[p.269]{Blo86}. %
\end{proof}

\begin{corollary} \label{coro:fieldChowGroups}
Let $k$ be a (possibly infinite) algebraic extension of $\F_p$, and $A$ a commutative $\Zop$-algebra. Then
\begin{equation}
 \Hom{\H(k, A)}(\un, \un(i)[j]) = \left \{ \begin{array}{ll}
A & \textrm{ if } i = j = 0,\\
0 & \textrm{ otherwise. }
 \end{array} \right .
\end{equation}
\end{corollary}

\begin{proof}
Since $\Zop$ is a field, any $\Zop$-module is flat and so it suffices to prove the case $A = \Zop$. 

By Theorem~\ref{theorem:GL} and Corollary~\ref{coro:calculateHigherChowGroups} it suffices to prove that $\Hyp_\Nis^{i-n}(k, \kmp[n]) = 0$ unless $n = i = 0$. Fields have Nisnevich cohomological dimension zero %
so we only need to consider the case $n = i$.
Since $\Hyp_\Nis^{0}(k, \kmp[n]) = \kmp[n](k)$, it suffices to show that the Milnor $K$-theory has no $p$-torsion for $n > 0$. For finite fields, this property is due to Steinberg, \cite[Exam.1.5]{Mil69}, and arbitrary algebraic extension, it follows from the fact that Milnor $K$-theory commutes with filtered colimits.
\end{proof}

\begin{proposition}[Projective bundle formula]\label{prop:projectivebundleformula}
Let $A$ be a commutative $\Zop$-algebra, and let $E \to X$ be a vector bundle of dimension $n$ in $\QProj / k$, and $p: \P(E) \to X$ its associated projective bundle. Then we have the following isomorphisms in $\H(X, A)$ 
\[ p_*\un_\P(E) \cong \bigoplus_{i = 0}^{n-1} \un(-i)[-2i], \qquad M(\P(E)) \cong \bigoplus_{i = 0}^{n-1} \un(i)[2i]. \]
\end{proposition}

\begin{proof}
For ease of notation, set $P = \P(E)$. Applying $\iHom{}(-, \un)$ to the right isomorphism gives the left one, so it suffices to prove the right one (note that $M(P) \cong p_\#p^*\un_X$ and $\iHom{}(p_\# p^*\un, \un) \cong \iHom{}(\un, p_*p^*\un)$). Now the right isomorphism in $\H(X, A)$ is the image of this isomorphism in $\H(X, \Zop)$ under the canonical functor $\H(X, \Zop) \to \H(X, A)$, so it suffices to treat the case $A = \Zop$.

Now the standard proof applies, cf. \cite[Prop.3.5.1]{Voev00}, \cite[Thm.4.2.7]{Voev96}. 
\end{proof}

\begin{remark}
Many of the other proofs from \cite{Voev00} also work in our setting. See \cite[Section 2.2]{Voev00} for a list. Note that in our setting, $f_*f^!\un_S$ plays the rôle of $M^c(X)$ for $f: X \to S$ in $\Sm/S$, $S \in \QProj/k$.
\end{remark}

\section{Mixed Stratified Tate Motives} \label{sec:mstm}
In the following we always assume that our varieties are defined over $\overline{\F}_p$ and that $\kk$ is a field of characteristic $p$. We will often drop the $\kk$ from the notation.
In this section we define the category of stratified mixed Tate motives as a full subcategory of $H(X,\K\otimes\kk)=H(X,\kk)=H(X)$ as constructed in the last section. We will then consider a weight structure on this category, prove a formality result, and state the Erweiterungssatz.
\subsection{Stratified mixed Tate motives}
Let $(X,\Ss)$ be an affinely stratified variety over $\overline{\F}_p$, i.e. a variety X with a finite partition into locally closed subvarieties (called the strata of $X$)
\begin{equation}
X=\bigcup_{s\in \Ss} X_s,
\end{equation}
such that each stratum $X_s$ is isomorphic to $\A^n$ for some $n$, and the closure $\overline{X}_s$ is a union of strata. The embeddings are denoted by $j_s: X_s\hookrightarrow X$. The prime example we always have in mind here is the flag variety of a reductive group with its Bruhat stratification.
Starting from this datum, \cite{SoeWe} defines the category of \emph{stratified mixed Tate motives} on $X$, which we recall in this paragraph.
We start with the basic case of just one stratum.
\begin{definition}\label{def:mixedtateonstrata}
	For $X\cong\A^n$, denote by $\DT{X,\kk}=\DT{X}$ the full triangulated subcategory of $\H(X,\kk)$ generated by motives isomorphic to $\un_X(q)$ for $q\in \Z$. Recall that $\un_X$ denotes the tensor unit in $\H(X,\kk)$.
\end{definition}
We shall make extensive use of the following statement.
\begin{proposition}\label{prop:mixedtateonstrata}
	For $X\cong\A^n$, there is an equivalence of monoidal $\kk$-linear categories
	\begin{equation}
\DT{X}\cong\kk\modules^{\Z\times\Z}\cong\Derb{\kk\modules^\Z}.
\end{equation}
Here, the $\kk\modules^{\Z\times\Z}$ denotes the category of bigraded, finite dimensional vector spaces over $\kk$ and $\Derb{\kk\modules^\Z}$ is the bounded derived category of graded, finite dimensional vector spaces over $\kk$.
	
	We choose the isomorphisms such that $\un_X(i)[j]$ corresponds to $\kk$ sitting in degree $(i,j)$ in $\kk\modules^{\Z\times\Z}$ and $\kk$ sitting in degree $i$ with respect to the grading and cohomological degree $-j$ in $\Derb{\kk\modules^\Z}$.
	
	Note that this equips $\DT{X}$ with a natural t-structure. We denote the $j$-th cohomology functor by 
	\begin{equation}
\pazocal H^j: \DT{X}\rightarrow \kk\modules^\Z.
\end{equation}

\end{proposition}
\begin{proof}
	Follows from Corollary \ref{coro:fieldChowGroups}.
\end{proof}
We can now proceed to the general case. Since our category should be closed under taking Verdier duals and other reasonable combinations of the six functors, we have to assume that $(X,\Ss)$ fulfils an additional condition:
\begin{definition}
	$(X,\Ss)$ is called \emph{Whitney--Tate} if and only if for all $s,t\in\Ss$ and $M\in \DT{X_s}$ we have $j_t^*j_{s*}M\in \DT{X_t}$.
\end{definition}

From now on we always assume that $(X,\Ss)$ is Whitney--Tate. In \cite{SoeWe} it is shown that (partial) flag varieties and other examples are indeed Whitney--Tate\footnote{In \cite[Prop{.}A.2]{SoeWe} Soergel--Wendt give a sufficient condition for a stratified scheme to be Whitney--Tate, involving the existence of certain resolutions of singularities of the closure of strata. In one step of the proof they use absolute purity, which is not proven for our formalism. But here relative purity actually suffices, and the proposition still applies in our setting.}.
\begin{definition}\label{def:stratmixedtate}
	The category of \emph{stratified mixed Tate motives} on $X$, denoted by $\DMT{\Ss}{X,\kk}=\DMT{\Ss}{X}$, is the full subcategory of $\H(X)$ consisting of objects $M$ such that $j_s^*M\in \DT{X_s}$ for all $s\in \Ss$.
\end{definition}
\begin{remark}\label{rem:WTec} Because we assumed $X$ to be Whitney--Tate, we could have also required the equivalent condition $j_s^!M\in \DT{X_s}$ for all $s\in \Ss$, see \cite[Lemma 4.4, Remark 4.8]{SoeWe}.
\end{remark}
\subsection{Affinely stratified maps}
The right definition of a map between affinely stratified varieties is different from the usual definition of a stratified map, as defined for example in \cite{GM88}. 
\begin{definition}
	Let $(X,\Ss)$ and $(Y,\Ss^\prime)$ be affinely stratified varieties. We call
	$f:X\rightarrow Y$ an \emph{affinely stratified map} if
	\begin{enumerate}
		\item for all $s \in\Ss^\prime$ the inverse image $f^{-1}(Y_s)$ is a union of strata;
		\item for each $X_s$ mapping into $Y_{s^\prime}$, the induced map $f:X_s\rightarrow Y_{s^\prime}$ is a surjective linear map.
	\end{enumerate}
\end{definition}
 Affinely stratified maps are defined in such a way that their associated pullback and pushforward functors preserve stratified mixed Tate motives.
\begin{lemma}\label{lem:sixfunctorsonstrata} Let $X\in \Sm/k$. Consider the maps
	\begin{equation}
s:X \rightleftarrows  \A^n_X:p
\end{equation}
where $p$ denotes the projection and $s$ the zero section. Then we can make the following identifications\\
	\begin{minipage}{0.5\linewidth}
		\begin{align*}
		p_*(\un_{\A^n_X})&=\un_{X}, 	\\
		p^*(\un_{X})&=\un_{\A^n_X}, 	\\
		s^*(\un_{\A^n_X})&=\un_{X} \text{ and}	
		\end{align*}
	\end{minipage}\hspace{-10pt}
\begin{minipage}{0.2\linewidth}
	\begin{align*}
	p_!(\un_{\A^n_X})&=\un_{X}(-n)[-2n],\\
	p^!(\un_{X})&=\un_{\A^n_X}(n)[2n],\\
	s^!(\un_{\A^n_X})&=\un_{X}(-n)[-2n]
	\end{align*}\end{minipage}\vspace{10pt}\\
Furthermore $\Dual_X(\un_{X}(m)[2m])=\un_{X}(\dim X-m)[2\dim X-2m]$ where $$\Dual_X = \iHom{X}(-,f^!(\un))$$ denotes the Verdier duality functor, for $f: X \to k$ the structural morphism, cf. Equation~\ref{equa:propertiesofH:dual}.
\end{lemma}
\begin{proposition}\label{prop:sixfunctorsandstratifiedmaps} 
Let $(X,\Ss)$ and $(Y,\Ss^\prime)$ be affinely Whitney--Tate stratified varieties and $f:X\rightarrow Y$ an affinely stratified map. Then the induced functors restrict to stratified mixed Tate motives on $X$ and $Y$
\begin{equation}
f_*,f_!: \DMT{\Ss}{X}\rightleftarrows\DMT{\Ss^\prime}{Y}:f^*,f^!.
\end{equation}
Also the internal Hom, duality and tensor product restrict.
\end{proposition}
\begin{proof}
	Duality preserves stratified mixed Tate motives because $(X,\Ss)$ and $(Y,\Ss^\prime)$ are Whitney--Tate; this follows from Remark \ref{rem:WTec}. So we only have to prove the statements for half of the six functors, cf. Theorem~\ref{theo:propertiesofH}\eqref{theo:propertiesofH:dual}. 
	The statement for $f^*$ follows directly from the definitions.
	
	We consider $f_!$ next. Let $E\in\DMT{\Ss}{X}$. We have to show $v^*f_!E\in\DT{Y_s}$ for all strata $v:Y_s\hookrightarrow Y$. By  base change applied to the cartesian diagram
	\begin{center}
		\cartesiandiagramwithmaps{f^{-1}(Y_s)}{w}{X}{g}{f}{Y_s}{v}{Y}
	\end{center}
	we have to show that $g_!w^*E\in\DT{Y_s}$. This can be done by an induction on the number of strata in $f^{-1}(Y_s)$. 
	Denote by $j$ the inclusion of an open stratum $X_s$ in $f^{-1}(Y_s)$ and by $i$ the one of the complement. We obtain the distinguished triangle
	\begin{center}
		\disttriangle{g_!j_!j^*w^*E}{g_!w^*E}{g_!i_!i^*w^*E}.
	\end{center}
	Let us first consider the left hand side. By assumption we have $j^*w^*E=j^*_sE\in \DT{X_s}$. Since $f$ is an affinely stratified map, $gj$ is a projection $\A^n\times\A^m\rightarrow\A^n$ and by Lemma \ref{lem:sixfunctorsonstrata} $(gj)_!$ maps $\DT{X_s}$ to $\DT{Y_s}$. Hence we have $g_!j_!j^*w^*E\in\DT{Y_s}$.
	The right hand side is a stratified mixed Tate motive by induction. Now the statement follows from the fact that $\DT{Y_s}$ is closed under extensions.
	
	The statement for the tensor product follows immediately, since pullback is a tensor functor, and we are done. 
\end{proof}
\subsection{Weights}
Weight structures---as first considered in \cite{Bon}---provide a very concise framework for the powerful \emph{yoga of weights}, as applied, for example, in the proof of the Weil conjectures or the decomposition theorem for perverse sheaves.
\begin{definition}
	Let $\pazocal C$ be a triangulated category. A \emph{weight structure} on $\pazocal C$ is a pair $(\pazocal C_{w\leq 0},\pazocal C_{w\geq 0})$ of full subcategories of $\pazocal C$ such that with $\pazocal C_{w\leq n}:=\pazocal C_{w\leq 0}[n]$ and $\pazocal C_{w\geq n}:=\pazocal C_{w\geq 0}[n]$ the following conditions are satisfied:
	\begin{enumerate}
		\item $\pazocal C_{w\leq 0}$ and $\pazocal C_{w\geq 0}$ are closed under direct summands;
		\item $\pazocal C_{w\leq 0}\subseteq \pazocal C_{w\leq 1}$ and $\pazocal C_{w\geq 1}\subseteq \pazocal C_{w\geq 0};$
		\item for all $X\in \pazocal C_{w\leq 0}$ and $Y\in\pazocal C_{w\geq 1}$, we have $\Hom{\pazocal C}(X,Y)=0;$
		\item for any $X\in \pazocal C$ there is a distinguished triangle \disttriangle{A}{X}{B} with $A\in \pazocal C_{w\leq 0}$ and $B\in \pazocal C_{w\geq 1}.$
	\end{enumerate}
	The full subcategory $\pazocal C_{w=0}=\pazocal C_{w\leq 0}\cap\pazocal C_{w\geq 0}$ is called the \emph{heart of the weight struture}.
\end{definition}
Unlike in the setting of motives with rational coefficients as considered in \cite{SoeWe}, a priori, it is not known if our category of motives can be equipped with a weight structure (the standard proofs rely on the existence of some kind of resolution of singularities and do not work for torsion coefficients equal to the characteristic of the base). Nevertheless we can define a weight structure directly on the category of stratified mixed Tate motives and also prove compatibilities with the six functors (at least for affinely stratified maps).

We start by defining weight structures for the Tate motives on the affine strata. Here we want $\un_{\A^n}(p)[q]$ to have weight $q-2p$. 
\begin{definition}
	Let $\DT{\A^n}_{w\leq 0}$ (resp. $\DT{\A^n}_{w\geq 0}$) be the full subcategory of $\DT{\A^n}$ consisting of objects isomorphic to finite direct sums of $\un_{\A^n}(p)[q]$ for $q\leq 2p$ (resp. $q\geq 2p$). This defines a weight structure on $\DT{\A^n}$.
\end{definition}
\begin{proof}
We use Proposition \ref{prop:mixedtateonstrata} to identify $\DT{\A^n}$ with the derived category of graded vector spaces. Here the axioms of a weight structure are easily checked.
\end{proof}
We can now obtain a weight structure for stratified mixed Tate motives by glueing.
\begin{definition} Let $(X,\Ss)$ be an affinely Whitney--Tate stratified variety. Then we obtain a weight structure on $\DMT{\Ss}{X}$ by setting
	\begin{align*}
	\DMT{\Ss}{X}_{w\leq 0}&:=\left\{M\setline j_s^*M \in \DT{X_s}_{w\leq 0}\text{ for all } s\in \Ss \right\}\text{ and}\\
	\DMT{\Ss}{X}_{w\geq 0}&:=\left\{M\setline j_s^!M \in \DT{X_s}_{w\geq 0}\text{ for all } s\in \Ss\right\}.
	\end{align*}
	
\end{definition}
\begin{proof}
	\cite[Proposition 5.1]{SoeWe}.
\end{proof}
With this definition we have the following compatibilities with the six functors.
\begin{proposition}\label{prop:weightexactness}
	Let $(X,\Ss)$ and $(Y,\Ss^\prime)$ be affinely Whitney--Tate stratified varieties and $f:X\rightarrow Y$ an affinely stratified map. Then
	\begin{enumerate}
		\item the functors $f^*,f_!$ are weight left exact, i.e. they preserve $w\leq0$;
		\item the functors $f^!,f_*$ are weight right exact, i.e. they preserve $w\geq0$;
		\item the tensor product is weight left exact, i.e. restricts to
		\begin{equation}
\DMT{\Ss}{X}_{w\leq n}\times\DMT{\Ss}{X}_{w\leq m}\rightarrow\DMT{\Ss}{X}_{w\leq n+m};
\end{equation}
	\item Verdier duality reverses weights, i.e. restricts to
				\begin{equation}
\Dual_X: \DMT{\Ss}{X}_{w\leq n}^{\op}\rightarrow\DMT{\Ss}{X}_{w\geq -n};
\end{equation}
	\item the internal Hom functor $\iHom{X}$ is weight right exact, i.e. restricts to
			\begin{equation}
\DMT{\Ss}{X}_{w\leq n}^{\op}\times\DMT{\Ss}{X}_{w\geq m}\rightarrow\DMT{\Ss}{X}_{w\geq m-n};
\end{equation}
	\item For f smooth $f^!$ and $f^*$ are weight exact, i.e. weight left exact and weight left exact;
		\item For f proper $f_!$ and $f_*$ are weight exact;
		\item If X is smooth $\un_X(n)[2n]$ is of weight zero for all $n\in\Z$.
	\end{enumerate}
\end{proposition}
\begin{proof}
	Follows by the same arguments as in Proposition \ref{prop:sixfunctorsandstratifiedmaps} while using Lemma \ref{lem:sixfunctorsonstrata} to see that the pullbacks and pushforwards associated to projections $p:\A^n\times\A^m\rightarrow\A^n$ preserve weights.
\end{proof}	
\subsection{Tilting and Pointwise purity} 
In this and following sections we will show that our category of stratified mixed Tate motives can often be realized as bounded homotopy category of some additive subcategory, using a process called \emph{tilting.} Examples of such subcategories are weight zero motives---which turn out to parity motives in our applications---projective perverse motives or tilting perverse motives. For this we will use the following theorem.
\begin{theorem}[Tilting]\label{thm:tiltinggeneral} Let $(X,\Ss)$ be an affinely Whitney--Tate stratified variety.  Let $\pazocal C\subset \DMT{\Ss}{X}$ be an additive subcategory such that
\begin{enumerate}
	\item For all $E,F\in\pazocal C$, we have $\Hom{\H(X)}(E,F[n])=0 \text{ for all }n\neq0$.
	\item $\pazocal C$ generates $\DMT{\Ss}{X}$ as a triangulated category.
\end{enumerate} 
Then there is an equivalence of triangulated categories, called \emph{tilting}, 
$$\Hotb{\pazocal C}\stackrel{\sim}{\rightarrow} \DMT{\Ss}{X}.$$
Here $\Hotb{\pazocal C}$ denotes the bounded homotopy category of chain complexes of $\pazocal C$.
\end{theorem}
\begin{proof}
By construction, see Section \ref{defi:H},  our category of motives is defined as the quotient category $\H(X)=D((\K_{X}\otimes\kk)\modules) / \Ss_{\Htp, \Stb}.$ By Rappels~\ref{rappels:tricat2}\eqref{BousfieldLocalisationTriCat}, $\H(X)$ is hence equivalent to a full subcategory $D((\K_{X}\otimes\kk)\modules).$ Since $\DMT{\Ss}{X}$ is a full subcategory of $\H(X)$ it can hence be realized as a full subcategory of the derived category of the Grothendieck abelian category $\pazocal A= (\K_{X}\otimes\kk)\modules.$
For every object in $\pazocal C$, considered as an object in $D(\pazocal A)$, choose a representative in $C(\pazocal A)$ which is homotopy-injective. The collection $\pazocal T$ of these choices of representatives fulfills the following properties:
\begin{enumerate}
	\item For all $E,F\in\pazocal T$, we have $\Hom{\Hotub(\pazocal A)}(E,F[n])=\Hom{D(\pazocal A)}(E,F[n])$, since $F$ is homotopy-injective. Here $\Hotub(\pazocal A)$  denotes the \emph{unbounded} homotopy category of chain complexes in $\pazocal A.$
	\item For all $E,F\in\pazocal T$, we have $\Hom{D(\pazocal A)}(E,F[n])=0 \text{ for all }n\neq0$.
	\item $\pazocal T$ generates $\DMT{\Ss}{X}$ as a triangulated category.
\end{enumerate} 
A collection fulfilling those properties is often called a \emph{tilting collection} in the literature. The statement now follows from \cite[Proposition 10.1]{Ric89} or \cite[Theorem 1]{Kel93}, see \cite[Appendix B]{SoeWe} for a nice sketch of the proof.
\end{proof}
To show that weight zero motives fulfill the first assumption of this theorem, we need to impose an additional \emph{pointwise purity} property.
\begin{definition}
	Let $?\in\{*,!\}$. A motive $M\in\DMT{\Ss}{X}$ is called \emph{pointwise ?-pure} if for all $s\in\Ss$
	\begin{equation}
j^?_s M\in \DT{X_s}_{w=0}.
\end{equation}
If both conditions are satisfied, the motive is called \emph{pointwise pure}. 
\end{definition}
We list some compatibilities of the six functors with pointwise purity.
\begin{proposition}\label{prop:pointwisepurepropermap}
	Let $(X,\Ss)$ and $(Y,\Ss^\prime)$ be affinely Whitney--Tate stratified varieties and $f:X\rightarrow Y$ an affinely stratified map. Then
	\begin{enumerate}
		\item For f smooth $f^!$ and $f^*$ preserve pointwise purity;
		\item For f proper $f_!$ and $f_*$ preserve pointwise purity.
\end{enumerate}
\end{proposition}
\begin{proof} Using Proposition \ref{prop:weightexactness},
	(1) follows from $f^!=f^*(-d)[-2d]$, where $d$ is the relative dimension of $f$, and
	(2) follows from base change and $f_!=f_*$.
\end{proof}
We can then prove the following Lemma, which  crucially depends on the fact that there are no non-trivial extension between the Tate objects on $\overline{\F}_p$:
\begin{lemma} \label{lem:pointwisepuremorphism}Let $(X,\Ss)$ be an affinely Whitney--Tate stratified variety.
	Let $E,F\in \DMT{\Ss}{X}$ such that $E$ is pointwise $*$-pure of weight zero and $F$ is pointwise $!$-pure of weight zero. Then
		\begin{equation}
		\Hom{\H(X)}(E,F[a])=0.
		\end{equation}
		for all $a\neq 0$.
\end{lemma}
\begin{proof}
See also \cite[Corollary 6.3]{SoeWe}.
We proceed by induction on the number of strata. Denote by $j:\A^n =U\hookrightarrow X$ the inclusion of an open stratum in $X$ and by $i:Z\hookrightarrow X$ its closed complement.
Hence there is a distinguished triangle 
$$\disttriangle{j_!j^*E}{E}{i_!i^*E}$$
and $\Hom{\H(X)}(E,F[a])$ fits in an exact sequence
$$\les{\Hom{\H(Z)}(i^*E,i^!F[a])}{\Hom{\H(X)}(E,F[a])}{\Hom{\H(\A^n)}(j^*E,j^!F[a])}$$
where the right term vanishes using Proposition \ref{prop:mixedtateonstrata} (this is where we use that there are no non-trivial extension between the Tate objects) and the left hand term vanishes by induction. The statement follows.
\end{proof}

\begin{theorem}[Tilting for weight zero motives]\label{thm:tilting}
	Let $(X,\Ss)$ be an affinely Whitney--Tate stratified variety, such that all objects of $\DMT{\Ss}{X}_{w=0}$ are additionally pointwise pure. Then there is an equivalence of categories, called \emph{tilting},
	\begin{equation}
\Delta:\DMT{\Ss}{X}\stackrel{\sim}{\rightarrow}\Hotb{\DMT{\Ss}{X}_{w=0}}.
\end{equation}

\end{theorem}
\begin{proof}
As heart of a weight structure, $\DMT{\Ss}{X}_{w=0}$ generates $\DMT{\Ss}{X}$ as a triangulated category. Together with Lemma \ref{lem:pointwisepuremorphism} this allows us to apply
Theorem \ref{thm:tiltinggeneral} and the statement follows.
\end{proof}
We want to remark that the tilting equivalence can also be stated in the form of a \emph{formality} theorem.
\begin{theorem}[Formality]  \label{theo:formality}
The last theorem can also be stated as natural equivalence
	\begin{equation}
\DMT{(\Ss)}{X}\stackrel{\sim}{\rightarrow}\operatorname{dgDer-}(E,d=0)
\end{equation}
where the right hand side denotes the dg-derived category of graded modules of the formal (equipped with a trivial differential) graded dg-algebra \begin{equation}
E=\bigoplus_{i,j\in\Z}\Hom{H(X)}(L,L(i)[j])
\end{equation}
where $L$ is a direct sum of objects $L_i$ generating $\DMT{\Ss}{X}_{w=0}$ with respect to direct sum, shifts,  $(n)[2n]$,  and isomorphism.
\end{theorem}
\subsection{Erweiterungssatz} \label{subsec:Erweiterungssatz}
The \emph{Erweiterungssatz} as first stated in \cite{Soe90} and reproven in a more general setting in \cite{Gin} allows a \emph{combinatorial} description of pointwise pure weight zero sheaves on $X$ in terms of certain modules over the cohomology ring of $X$. In the case of $X$ being the flag variety, these modules are called \emph{Soergel modules}.

In our setting, the same results hold by replacing the usual singular or étale cohomology ring by the motivic one.
\begin{definition} Let $X\in \QProj / k$ and $E\in \H(X)$.
	Denote by 
	\begin{equation}
\Hyp(E):=\bigoplus_{(i,j)\in\Z\times\Z}\Hom{\H(X)}(\un_X,E(j)[i])
\end{equation}
the \emph{hypercohomology} functor. 
Note that $\Hyp(E)$ is naturally a bigraded right module over $\Hyp(X,\kk):=\Hyp(X):=\Hyp(\un_X)= H^{\bullet}_\mathcal{M}(X, \kk (\bullet))$, the motivic cohomology ring of $X$, where $H^{i}_\mathcal{M}(X, \kk (j))=\Hom{\H(X)}(\un_X,\un_X(j)[i])$ sits in bidegree $(i,j).$
\end{definition}

Motivic cohomology (at least for smooth schemes) can be identified with higher Chow groups, which are usually quite infeasible for computation. For affinely stratified varieties everything gets much easier:
\begin{theorem}\label{thm:chowringcellular}
	Let $(X,\Ss)$ be an affinely stratified and irreducible variety of dimension $n$ over $k$. Then:
	\begin{enumerate}
		\item The motivic cohomology ring $\Hyp(X)= H^{\bullet}_\mathcal{M}(X, \kk (\bullet))$ is concentrated in degrees (2i,i) and we have
		\begin{equation}
H^{2i}_\mathcal{M}(X, \kk (i))=\bigoplus_{\substack{s\in \Ss,\\\dim X_s=i}}\kk.
\end{equation}
		\item If $X$ is furthermore smooth, there are graded $\kk$-algebra isomorphisms
		\begin{equation}
\Hyp(X)\cong CH^\bullet(X,\kk)\cong CH^\bullet(X,\Z)\otimes\kk,
\end{equation}
where for a ring $\Lambda$ we denote by $\Chow^\bullet(X,\Lambda)=\Chow_{n-\bullet}(X,\Lambda)$ the classical Chow ring of $X$ with coefficients in $\Lambda$, and $H^{2i}_\mathcal{M}(X, \kk (i))$ corresponds to $CH^i(X,\kk)$ under the isomorphism.
\end{enumerate}
	
\end{theorem}
\begin{proof}
	(1) follows by a standard argument using an induction on the number of strata, which we quickly repeat here. Denote by $j:\A^n =U\hookrightarrow X$ the inclusion of an open stratum in $X$ and by $i:Z\hookrightarrow X$ its closed complement. Let $p:X\rightarrow\Spec(k)$ be the structure map. Then by the localization property we have the distinguished triangle in $\Derb{\kk\mod}$
	\begin{equation}
\disttriangle{p_*i_*i^!\un_X}{p_*\un_X}{p_*j_*j^*\un_X}
\end{equation}
For the right hand side we have 
\begin{equation}
p_*j_*j^*\un_X=p_*j_*\un_U=\un_{\Spec(k)}.
\end{equation}
 
	Denote by $d$ the codimension of $Z$ in $X$, then by relative purity we have  \begin{equation}
p_*i_*i^!\un_X=p_*i_*\un_Z(-d)[-2d]
\end{equation}
for the left hand side and its cohomology is (up to the shift / twist) the motivic cohomology of $Z$, for which we can apply the induction hypothesis. In particular, the left hand side $p_*i_*i^!\un_X$ is concentrated in even cohomological degrees and the statement follows by the cohomology long exact sequence associated to the distinguished triangle together with Lemma \ref{prop:mixedtateonstrata}.\\
	(2) The first equality follows from (1) and Corollary \ref{coro:calculateHigherChowGroups} (here we need to assume that $X$ is smooth). The second equality follows since the Chow groups of an affinely stratified variety are indeed free. One can see this by using the same arguments as in (1) and the localization property for higher Chow groups, see \cite{Blo86} and \cite{Lev94}.
\end{proof}

So in our setting the Erweiterungssatz reads:

\begin{theorem}[Erweiterungssatz]\label{thm:Erweiterungssatz} Let $(X,\Ss)$ be an affinely Whitney--Tate stratified and proper variety and $E,F\in \DMT{\Ss}{X}$ pointwise pure. Assume additionally that for each embedding $j$ of a stratum $\Hyp E\rightarrow \Hyp j_*j^*E$ is surjective and
 $\Hyp j_!j^! F\rightarrow \Hyp F$ is injective. Then hypercohomology induces an isomorphism
 \begin{equation}
\Hom{\H(X)}(E,F)\stackrel{\sim}{\rightarrow}\Hom{\Hyp(X)}^{\Z\times\Z}(\Hyp(E),\Hyp(F))
\end{equation}
where the right hand denotes morphisms of bigraded $\Hyp(X)$-modules.
\end{theorem}
\begin{proof}
	\cite[Theorem 8.4]{SoeWe}.
\end{proof}

\begin{corollary} \label{coro:pwPureErweiterungssatz}
 Let $(X,\Ss)$ be an affinely Whitney--Tate stratified and proper variety and assume that	all objects in $\DMT{\Ss}{X}_{w=0}$ satisfy the conditions of Theorem \ref{thm:Erweiterungssatz}. Then hypercohomology induces a fully faithful embedding
 \begin{equation}
\Hyp:\DMT{\Ss}{X}_{w=0}\rightarrow \operatorname{mod}^{\Z\times\Z}\operatorname{-} \Hyp(X)
\end{equation}
of weight zero motives into the category of bigraded $\Hyp(X)$-modules. We denote the essential image by $\operatorname{Smod}^{\Z\times\Z}\operatorname{-} \Hyp(X)$.
\end{corollary}

\begin{remark} For $X$  affinely stratified, $\Hyp(X)$ is concentrated in degrees $(2i,i)$ by Theorem \ref{thm:chowringcellular}. By similar arguments the hypercohomology of all pointwise pure weight zero motives will also only live in degrees $(2i,i)$, hence in fact we have a fully faithful embedding
		\begin{equation}
\Hyp:\DMT{\Ss}{X}_{w=0}\rightarrow \operatorname{mod}^{\Z}\operatorname{-} \Hyp(X)
\end{equation}
into the category of $\Z$-graded modules, with respect to the diagonal grading.
\end{remark}
\section{Parity motives, perverse motives and the flag variety} \label{sec:parityMotives}
In this section we want to study certain interesting subcategories of our category of stratified mixed Tate motives. The general principle here is that everything works as one is used to from constructible étale sheaves or mixed Hodge modules. As in the last section, all varieties are over $\overline{\F}_p$ and $\kk$ is an arbitrary field of characteristic $p$, which we will often drop from the notation.

\subsection{Parity motives} \label{sec:parity}
Parity sheaves were first used by \cite{Soe00} as a substitute for intersection complexes in a setting where the decomposition theorem (c.f. \cite{BBD}, \cite{Saito}, \cite{CM}) does not hold in general, namely for constructible sheaves with modular coefficients. Then \cite{JMW} properly axiomatized and classified them. In this section we want to recall their properties and argue that the whole theory works fine in the setting of motives. All one really needs is a six functor formalism, as already stated in the introduction of \cite{JMW}. We want to remark that our situation is simpler, since we just consider affine strata and hence only trivial local systems. Furthermore,  \cite{JMW} makes extensive use of the fact that the surrounding constructible derived category of sheaves is Krull--Remak--Schmidt, meaning that every object is isomorphic to a finite direct sum of objects, each of which has local endomorphism ring. This is also the case for stratified mixed Tate motives.
\begin{lemma} Let $(X,\Ss)$ be an affinely Whitney--Tate stratified variety. Then $\DMT{\Ss}{X}$ is Krull-Remak-Schmidt.
\end{lemma}
\begin{proof} First of all $\DMT{\Ss}{X}$ is idempotent complete by \cite{LE2007452}, since it admits a bounded $t$-structure (for example the perverse $t$-structure defined in Section \ref{subsec:perversemotives}). Now $\End{\H(X)}(M)$ is a finitely generated $\kk$-algebra for every $M\in \DMT{\Ss}{X},$ as an inductive argument on the number of strata and the localisation long exact sequence easily shows.  Now \cite[Corollary A.2]{doi:10.1080/00927870701649184} shows that $\DMT{\Ss}{X}$ is Krull--Remak--Schmidt.
\end{proof}

\begin{definition} Let $(X,\Ss)$ be an affinely Whitney--Tate stratified variety.
A motive $E \in \DMT{\Ss}{X}$ is called \emph{parity} if it can be decomposed into a direct sum $E_1 \oplus E_2$ such that for all $s \in \Ss$ and $?\in\{!,*\}$ we have
\begin{align*}
\pazocal H^i(j^?_s E_k)=0\text{ if }i \not\equiv k\text{ mod }2.
\end{align*}
We denote the full subcategory of motives which are parity by \begin{equation}
\Par{\Ss}{X,\kk}=\Par{\Ss}{X}\subseteq\DMT{\Ss}{X}.
\end{equation}
\end{definition}

As opposed to intersection complexes, the existence of parity sheaves/motives with prescribed support is not known in general, only their uniqueness.
\begin{theorem}\label{thm:uniquenessparitysheaves}
	For all $s\in \Ss$ there exists (up to isomorphism) at most one indecomposable parity motive $E_s$ supported on $\overline{X}_s$ with $j^*_s E_s=\un_{X_s}$.
\end{theorem}
\begin{proof}
\cite[Theorem 2.12]{JMW} translates unchanged.
\end{proof}
The following theorem is a useful tool for constructing parity motives. It can be thought of as an analogue of the decomposition theorem.
\begin{proposition}\label{prop:propermapsparity}
 	Let $(X,\Ss)$ and $(Y,\Ss^\prime)$ be affinely Whitney--Tate stratified varieties and $f:X\rightarrow Y$ a proper affinely stratified map. Then $f_!=f_*$ preserves the parity condition.
\end{proposition}
\begin{proof}
This is proven in \cite[Prop. 2.34]{JMW}. Their condition of $f$ being even is not needed/trivially fulfilled in our setting. One can also show the statement analogously to the proof of Proposition \ref{prop:sixfunctorsandstratifiedmaps}.
%
\end{proof}
Under the condition that the closures of all strata admit proper resolutions we can identify the additive category of finite direct sums of (appropriately shifted and twisted) parity motives with the category of weight zero motives.

\begin{theorem}\label{thm:paritymotives}
	Let $(X,\Ss)$ be an affinely Whitney--Tate stratified variety and assume that for every $s\in\Ss$ there exists a proper, affinely stratified map
	\begin{equation}
\pi: \widetilde{X}_s\rightarrow\overline{X}_s\subset X
\end{equation}
	with $\widetilde{X}_s$ smooth, inducing an isomorphism over $X_s$.
	Then for every $s\in\Ss$ there exists an indecomposable, pointwise pure parity motive $E_s\in \Par{\Ss}{X}$ with $j_s^*E_s=\un_{X_s}$. The objects in $\Par{\Ss}{X}_{w=0}$ are the motives isomorphic to finite direct sums of the $E_s(n)[2n]$ for $n\in\Z$, $s\in\Ss$ and we have
	\begin{equation}
\DMT{\Ss}{X}_{w=0}=\Par{\Ss}{X}_{w=0}.
\end{equation}
	Furthermore every parity motive is a direct sum of motives in $\Par{\Ss}{X}_{w=0}$.
\end{theorem}
\begin{proof}
	Since $\widetilde{X}_s$ is smooth, the dual of $\un_{\widetilde{X}_s}$ is $\Dual_{\widetilde{X}_s}\un_{\widetilde{X}_s}=\un_{\widetilde{X}_s}(\dim X_s)[2\dim X_s]$. Hence the restriction of $\un_{\widetilde{X}_s}$ to all strata of $\widetilde{X}_s$ using $!$ or $*$ is pure of weight zero and concentrated in even cohomological degrees. Hence $\un_{\widetilde{X}_s}$ is parity and pointwise pure of weight zero. Since $\pi$ is proper, $\pi_!\un_{\widetilde{X}_s}$ is also parity and pointwise pure of weight zero by Propositions \ref{prop:pointwisepurepropermap} and \ref{prop:propermapsparity}. Furthermore by base change $j_s^*\pi_!\un_{\widetilde{X}_s}=\un_{X_s}$. So we choose $E_s$ to be the unique indecomposable direct summand of $\pi_!\un_{\widetilde{X}_s}$ with $j_s^*E_s=\un_{X_s}$. 
	By Theorem $\ref{thm:uniquenessparitysheaves}$ we know that these are all weight zero indecomposable parity motives---up to shifting and twisting by $(n)[2n]$. 
	
	The other statements can be proven along the lines of \cite[Corollary 6.7]{SoeWe}. By a standard induction argument one sees that the $E_s(n)[2n]$ generate $\DMT{\Ss}{X}$ as a triangulated category. The pointwise purity and Proposition \ref{prop:pointwisepurepropermap} imply that $\Hom{\H (X)}(E,F[a])=0$ for all $a>0$ and $E,F\in\Par{\Ss}{X}_{w=0}$. By \cite[Proposition 1.7(6)]{Bon14} it follows that 
	\[\Par{\Ss}{X}_{w=0}=\DMT{\Ss}{X}_{w=0}\]
	which concludes the proof.
\end{proof}
The pointwise purity of the indecomposable parity motives allows us to apply the tilting result from the last section and we obtain:

\begin{corollary}\label{cor:paritymotivestilting}
	Under the assumptions of Theorem \ref{thm:paritymotives} there is an equivalence of categories
	\begin{equation}
\DMT{\Ss}{X}\cong\Hotb{\Par{\Ss}{X}_{w=0}}.
\end{equation}
\end{corollary}
\subsection{Perverse motives}\label{subsec:perversemotives}
The whole theory of perverse sheaves from \cite[\S1, \S2]{BBD} applies in our setting. Again, all one needs is a six functor formalism. In particular, we can perversely glue the standard $t$-structures on the categories of mixed Tate motives on the strata---recall that they are just derived categories of graded vector spaces---to obtain a perverse $t$-structure on the category of stratified mixed Tate motives on an affinely stratified variety. Our goal is to show that---under a technical assumption---the category of stratified mixed Tate motives can be realized as the derived category of perverse motives, the homotopy category of projective perverse motives, or the homotopy category of tilting perverse motives, respectively.
\begin{definition} 
	Let $(X,\Ss)$ be an affinely Whitney--Tate stratified variety. Then we obtain a  $t$-structure, called the \emph{perverse $t$-structure}, on $\DMT{\Ss}{X}$ by setting
	\begin{align*}
	\DMT{\Ss}{X}^{p\leq 0}&:=\left\{M\setline j_s^*M \in \DT{X_s}^{\leq-\dim X_s}\text{ for all } s\in \Ss \right\}\text{ and}\\
	\DMT{\Ss}{X}^{p\geq 0}&:=\left\{M\setline j_s^!M \in \DT{X_s}^{\geq-\dim X_s}\text{ for all } s\in \Ss\right\}.
	\end{align*}
	We denote the heart of this $t$-structure by 
	\begin{equation}
	\Per{\Ss}{X,\kk}=\Per{\Ss}{X}.
	\end{equation}
	 This is an abelian category and we call its objects \emph{perverse motives} on $X$.
\end{definition}
\begin{proposition}\label{prop:perverseexactness}
	Let $(X,\Ss)$ be an affinely Whitney--Tate stratified variety and $j:W\rightarrow X$ be an inclusion of a union of strata. Then
	\begin{enumerate}
		\item the functors $j^*,j_!$ are right t-exact, i.e. they preserve $p\leq0$;
		\item the functors $j^!,j_*$  are left t-exact, i.e. they preserve $p\geq0$;
		\item the tensor product is weight left exact, i.e. restricts to
		\begin{equation}
\DMT{\Ss}{X}^{p\leq n}\times\DMT{\Ss}{X}^{p\leq m}\rightarrow\DMT{\Ss}{X}^{p\leq n+m};
\end{equation}
		\item Verdier duality reverses the $t$-structure, i.e. restricts to
		\begin{equation}
\Dual_X: \DMT{\Ss}{X}^{p\leq n,\op}\rightarrow\DMT{\Ss}{X}^{p\geq -n};
\end{equation}
		\item the internal Hom functor $\iHom{X}$ is weight right t-exact, i.e. restricts to
		\begin{equation}
\DMT{\Ss}{X}^{p\leq n,\op}\times\DMT{\Ss}{X}^{p\geq m}\rightarrow\DMT{\Ss}{X}^{p\geq m-n};
\end{equation}
		\item For $j$ smooth $j^!$ and $j^*$ are t-exact;
		\item For $j$ proper $j_!$ and $j_*$ are t-exact.
	\end{enumerate}
\end{proposition}
\begin{proof}
	See \cite[Propositions 2.1.6 and 2.1.20]{BBD}.
\end{proof}
\begin{definition} Let $(X,\Ss)$ be an affinely Whitney--Tate stratified variety. Then   $$\Delta_s:=j_{s,!}\un_{X_s}[\dim X_s]\in\DMT{\Ss}{X}$$ is called a \emph{standard object} and
 $$\nabla_s:=j_{s,*}\un_{X_s}[\dim X_s]\in\DMT{\Ss}{X}$$ is called a \emph{costandard object.} 
\end{definition} 
Here we encounter a technical difficulty. For étale sheaves, Artin's vanishing theorem \cite[XIV, Theorem 3.1]{SGA43} implies that affine maps are exact with respect to the perverse $t$-structure. In particular for an affinely stratified variety the standard and costandard objects in the category étale sheaves are perverse. There is no motivic proof of this fact yet, and hence we have to make this an additional assumption in the following statements.  Let $(X,\Ss)$ be an affinely Whitney--Tate stratified variety, we will often assume:
\begin{enumerate}
\item[\dagg]  All standard objects or equivalently all costandard objects are perverse, i.e., $\nabla_s, \Delta_s\in\Per{\Ss}{X}$ for all $s\in\Ss.$
\end{enumerate}
We will show that for example flag varieties fulfill \dagg.
\begin{definition}
Let $(X,\Ss)$ be an affinely Whitney--Tate stratified variety fulfilling \dagg.
We say that $E\in\Per{\Ss}{X}$ has \begin{enumerate}
	\item a \emph{standard flag} if $E$ has a filtration whose subquotients are standard objects $\Delta_s(a)$ for $s\in\Ss$, $a\in \Z$ or 
	\item a \emph{costandard flag} if $E$ has a filtration whose subquotients are costandard objects $\nabla_s(a)$ for $s\in\Ss$, $a\in \Z$.
\end{enumerate}
\end{definition}
\begin{proposition}\label{prop:enoughprojperversemotives}
	Let $(X,\Ss)$ be an affinely Whitney--Tate stratified variety fulfilling \dagg. Then $\Per{\Ss}{X}$ has enough projective and injective objects. Furthermore the projective objects have a standard flag and the injective objects have a costandard flag.
\end{proposition}
\begin{proof}
	See \cite[Theorem 3.2.1]{BGS} and \cite[Proposition 11.7]{SoeWe}.
\end{proof}
\begin{lemma}\label{lem:standardcostandardtilting}
	Let $(X,\Ss)$ be an affinely Whitney--Tate stratified variety fulfilling \dagg. Let $E,F\in \Per{\Ss}{X}$ such that $E$ has a standard flag and $F$ has a costandard flag. Then for all $n\neq 0$ we have
	$\Hom{\H(X)}(E,F[n])=0$.
\end{lemma}
\begin{proof}
		See for example \cite[Lemma 11.8, Theorem 11.10]{SoeWe}.
\end{proof}
These statements allow us to apply tilting.
\begin{theorem}\label{thm:perversemotivestilting}
	Let $(X,\Ss)$ be an affinely Whitney--Tate stratified variety fulfilling \dagg. Then tilting induces the following equivalences of categories
	\begin{equation}
\Derb{\Per{\Ss}{X}}\stackrel{\sim}{\rightarrow}\Hotb{\Projectives(\Per{\Ss}{X})}\stackrel{\sim}{\rightarrow}\DMT{\Ss}{X}.
\end{equation}
\end{theorem}
\begin{proof}
	Using the same argument as in \cite[Theorem 11.10]{SoeWe}, one sees that Lemma \ref{lem:standardcostandardtilting} and Propostion \ref{prop:enoughprojperversemotives} imply that category of projective perverse motives $\Projectives(\Per{\Ss}{X})$ fulfills the assumption of Theorem \ref{thm:tiltinggeneral} and the statement follows.
\end{proof}
In the preceding theorem we can replace projective objects by so called tilting objects. For a nice reference see \cite{BBM}.
\begin{definition}
Let $(X,\Ss)$ be an affinely Whitney--Tate stratified variety fulfilling \dagg. An object $E\in \Per{\Ss}{X}$ is called \emph{tilting} if it has both a standard flag and costandard flag. We denote the additive subcategory of \emph{tilting perverse motives} by $\Tilt(\Per{\Ss}{X})$.
\end{definition}
\begin{proposition}\label{prop:enoughtiltingperversemotives}
Let $(X,\Ss)$ be an affinely Whitney--Tate stratified variety fulfilling \dagg. Then for every stratum there exists a unique tilting perverse sheaf $T_s$ supported on $\overline{X}_s$ with $j_s^*T_s=\un_{X_s}[\dim X_s]$.
\end{proposition}
\begin{proof}
This follows from standard theory of highest weight categories and Proposition \ref{prop:enoughprojperversemotives} and Lemma \ref{lem:standardcostandardtilting}.
See for example \cite[Section 4/5]{Rin}.
\end{proof}
Again we can apply tilting to obtain:
\begin{theorem}
	Let $(X,\Ss)$ be an affinely Whitney--Tate stratified variety fulfilling \dagg. Then tilting induces an equivalence of categories
	\begin{equation}
\Hotb{\Tilt(\Per{\Ss}{X})}\stackrel{\sim}{\rightarrow}\DMT{\Ss}{X}.
\end{equation}
\end{theorem}
\begin{proof}
Lemma \ref{lem:standardcostandardtilting} and Propostion \ref{prop:enoughtiltingperversemotives} imply that $\Tilt(\Per{\Ss}{X})$ fulfills the assumption of Theorem \ref{thm:tiltinggeneral} and the statement follows.
\end{proof}
\subsection{Example: The flag variety}\label{subsec:flagvar} Denote by $G\supset B \supset T$ a split reductive algebraic group over $\overline{\F}_p$ with a Borel subgroup $B$ and maximal torus $T$. Denote by $X(T)\supset \Phi\supset \Phi_+$ the character lattice, root system, and positive roots associated to $B\supset T$. Let $X=G/B$ denote the \emph{flag variety} of $G$. Write $\Weyl=\operatorname{N}_G(T)/T$ for the Weyl group with simple reflections $\Simp$ and $X_w=BwB/B$ for the Bruhat cell for $w\in\Weyl$. 
Then the Bruhat decomposition \begin{equation}
X=\bigcup_{w\in \Weyl} X_w,
\end{equation}
gives rise to an affine Whitney--Tate stratification, which is denoted $(X,(B))$.

In the following, we will study the category $\DMT{(B)}{X}$ and certain interesting subcategories of it. 
\subsubsection{Parity motives and Soergel modules} \label{subsubsec:parmotsoemod}
We start by describing some motivic cohomology rings.
To the root system $\Phi$ of $G$ one associates a positive integer $t_\Phi$, called the \emph{torsion index}, see \cite{Dem73} and \cite{Gro58}. The torsion index $t_\Phi$ is a product of primes associated to the simple constituents of $\Phi$, which can be found the table
 \begin{center}
 	\begin{tabular}{c|c|c|c}
 	$A_l$, $C_l$ &  $B_l$ {\scriptsize($l\geq3$)}, $D_l$ {\scriptsize($l\geq4$)}, $G_2$ & $E_6$, $E_7$, $F_4$ & $E_8$ \\\hline
 	1 & 2 & 2,3 & 2,3,5\\
 \end{tabular}.
 \end{center}
 For an arbitrary $\Z$-module $M$, denote by $\SymA(M)$ its symmetric algebra, where we by convention put $M$ in degree 2. The \emph{coinvariant algebra} is defined by 
 \begin{equation}
C\stackrel{def}{=} \biggl ( \SymA(X(T))/\SymA(X(T))^\Weyl_+ \biggr )\otimes \kk,
\end{equation} 
where $X(T)=\Hom{\Sch(\overline{\F}_p)}(T,\G_m)\cong\Z^{\operatorname{rank}(T)}$ denotes the character lattice and $\SymA(X(T))^\Weyl_+$ denotes $\Weyl$-invariant elements of degree greater than zero. 
\begin{theorem}\label{thm:chowflag} Assume that $t_\Phi$ is invertible in $\kk$. 
	Then there is an isomorphism of graded $\kk$-algebras
	\begin{equation}
C \stackrel{\sim}{\to} \Hyp (X,\kk).
\end{equation}
Let $s\in\Simp$ be a simple reflection and $P_s=BsB\cup B$ the associated minimal parabolic. Then furthermore
\begin{equation}
C^s \cong \Hyp (G/P_s,\kk).
\end{equation}
\end{theorem}
\begin{proof}
	Since $X$ is an affinely stratified smooth variety we can use Theorem \ref{thm:chowringcellular} to identify $\Hyp(X)$ with the classical Chow ring of $X$ with coefficients in $\kk$. Under our hypotheses, \cite[Section 8]{Dem73} and \cite{Dem74}
	shows that the first Chern class \begin{equation}
X(T)\rightarrow\operatorname{Pic}X\stackrel{c_1}{\rightarrow} CH^1(X)
\end{equation} induces the claimed isomorphism.
	
	For the second statement, again $\Hyp (G/P_s)$ can be identified with the classical Chow ring. Denote by $L\subset P$ the Levi subgroup of $P=P_s$. Then by \cite[Prop.3.4 and Cor.5.9]{Kri13} we have the following chain of equalities
	\begin{equation}
\Chow^*(G/P)=\Chow^*_P(G)=\Chow^*_L(G)=(\Chow^*_T(G))^{\Weyl_L}=(\Chow^*(G/B))^{\Weyl_L},
\end{equation}
where the subscript denotes equivariant Chow groups and $\Weyl_L=\{1,s\}$ is the Weyl group of $L$.
\end{proof}
 Our next goal is to understand $\DMT{(B)}{X}_{w=0}$ using the ideas and results of \cite{Soe00}.
For $w\in \Weyl$, let $w=s_1\dots s_l$ with $s_i\in\Simp$ be a reduced expression which we will write as $\underline w=(s_1,\dots,s_l)$. Recall the \emph{Bott--Samelson resolution} of the \emph{Schubert variety} $\overline{X}_w$ given by 
\begin{equation}
\pi_{\underline{w}}:\BS(\underline{w})\stackrel{def}{=}P_{s_1}\times^B\dots\times^BP_{s_l}/B \rightarrow \overline{X}_w\subset X,
\end{equation}
where the morphism is given by multiplication. The variety $\BS(\underline{w})$ is smooth and $\pi_{\underline{w}}$ is proper and induces an isomorphism on $X_w$. Hence, we can apply Theorem \ref{thm:paritymotives} and the Erweiterungssatz (Theorem \ref{thm:Erweiterungssatz}) to identify weight zero motives, weight zero parity motives and Soergel modules
\begin{align*}
\DMT{(B)}{X,\kk}_{w=0}&=\Par{(B)}{X,\kk}_{w=0}\\
&=\left\langle\, \pi_{\underline{w},!}\un_{\BS(\underline{w})}(n)[2n]\,\vert\, w\in\Weyl, n \in \Z\,\right\rangle_{\oplus,\inplus,\cong}\\
&\cong \left\langle\,  \Hyp\left(\pi_{\underline{w},!}\un_{\BS(\underline{w})}(n)[2n]\right)\,\vert\, w\in\Weyl, n \in \Z\,\right\rangle_{\oplus,\inplus,\cong}\\
&\subset \operatorname{mod}^\Z\operatorname{-}\Hyp(X,\kk).
\end{align*}
Here $\oplus$, $\inplus$ and $\cong$ means closure under finite direct sums, direct summands and isomorphisms in the category of motives and $\Hyp(X)$-modules, respectively.

As a last step we want to recall Soergel's explicit description of the \emph{Bott--Samelson modules} $\Hyp\left(\pi_{\underline{w},!}\un_{\BS(\underline{w})}\right)$.
For $s\in S$ denote by $\pi_s:X=G/B \rightarrow G/P_s$ the projection. 
Then by \cite[Lemma 3.2.1]{Soe00} we have
\begin{equation}
\pi_{\underline{w},!}\un_{\BS(\underline{w})}=\pi_{s_l}^*\pi_{s_l,*}\cdots\pi_{s_1}^*\pi_{s_1,*}\un_{B/B}.
\end{equation}
Hence we need to understand the interaction of the functors $\Hyp$ and $\pi_{s}^*\pi_{s,*}$.
\begin{lemma}Let $s\in S$ be a simple reflection.
	Then we can identify $$\pi_{s,*}\un_{G/B}=\un_{G/P_s}\oplus\un_{G/P_s}(-1)[-2].$$
	Assume that $t_\Phi$ is invertible in $\kk$. Then there is a natural equivalence of functors 
	$$\Hyp(\pi_{s_*}\pi_s^*(-))\cong C\otimes_{C^s}\Hyp(-): \DMT{(B)}{X}_{w=0}\rightarrow C\modules^\Z.$$
\end{lemma}
\begin{proof}
	There are two different proofs for the first statement which do not apply in our situation. The first one given in \cite{Soe90} uses the decomposition theorem for perverse sheaves and the second one in \cite{Soe00} relies on a concrete description of the category of sheaves and the identification $G(\C)/B(\C)=K/T$ for a compact real form $K$ of the complex group $G(\C)$, while it also requires that $2$ is invertible in $\kk$. 
	But one can also apply the projective bundle formula (see Proposition \ref{prop:projectivebundleformula}). It is well known that $\pi_s: G/B\rightarrow G/P_s$ is a $\Proj^1$-bundle. By \cite[Exercise 7.10(c)]{Har} this bundle is the projectivization $\Proj(\pazocal E)$ of a vector bundle $\pazocal E$ on $G/P_s$, since $G/B$ is regular and Noetherian. In our case there is an completely explicit description of the bundle $\pazocal E$ for which we did not find an reference. Hence the projective bundle formula applies and the first statement follows.

Following \cite[Theorem 14]{Soe90} or \cite[Proposition 4.1.1]{Soe00} this implies that there is a natural isomorphism of functors
$$\Hyp(\pi_{s_*}\pi_s^*(-))\cong \Hyp(X)\otimes_{\Hyp(G/P_s)}\Hyp(-): \DMT{(B)}{X}\rightarrow \Hyp(X)\modules^\Z$$
and the statement follows using Theorem \ref{thm:chowflag}.
\end{proof}
Using $\Hyp(\un_{B/B})=\kk$ and applying the preceding Lemma we get an isomorphism
\begin{equation}
\Hyp\left(\pi_{\underline{w},!}\un_{\BS(\underline{w})}\right)\cong C\otimes_{C^{s_l}}\cdots C\otimes_{C^{s_1}}\kk.
\end{equation}
Furthermore, the pointwise purity of the motives $\pi_{\underline{w},!}\un_{\BS(\underline{w})}$ allows us to use the tilting result (Theorem \ref{thm:tilting}). So in conclusion we obtain the following theorem.
\begin{theorem}\label{thm:mainflag}
	There is an equivalence of categories
	\begin{equation}
\DMT{(B)}{X}=\Hotb{\Par{(B)}{X}_{w=0}}.
\end{equation}
	Assume that $t_\Phi$ is invertible in $\kk$. Then $\Par{(B)}{X}_{w=0}$ can be identified with the category of evenly graded Soergel modules 
	\begin{equation}
C\soergelmodules^\Z_{ev}=\left\langle\,  C\otimes_{C^{s_1}}\cdots C\otimes_{C^{s_n}}\kk\,\vert\, s_i\in\Simp\,\right\rangle_{\oplus,\inplus,\cong,\langle2-\rangle}
\end{equation}
	where $\oplus$, $\inplus$, $\cong$ and $\langle2-\rangle$ means closure under finite direct sums, direct summands, isomorphisms and even shifts of grading in the category of graded $C$-modules.
	
	Under this isomorphism the unique indecomposable parity motive $E_w$ with $j_w^*E_w=\un_{X_w}$ gets identified with the unique indecomposable Soergel module $D_w$ which appears as a direct summand of $C\otimes_{C^{s_1}}\cdots C\otimes_{C^{s_n}}\kk$ but not in the corresponding modules for smaller expressions.
\end{theorem}
\begin{remark}
	The equivalence $\Par{(B)}{X}_{w=0}\stackrel{\sim}{\rightarrow} C\soergelmodules^\Z_{ev}$ also proves that in case of the flag variety, the category of stratified mixed Tate motives is equivalent to the \emph{mixed derived category} as considered in \cite[Definition 2.1]{AR2}. Their mixed derived category is by construction the homotopy category of Soergel modules.
	\end{remark}
\subsubsection{Perverse motives} We start by showing that the technical requirements for a nice theory of perverse motives and tilting perverse motives are in fact met by the flag variety.
\begin{lemma}
The flag variety with its Bruhat stratification $(X,(B))$ fulfills \dagg, i.e., $\nabla_w, \Delta_w\in\Per{(B)}{X}$ for all $w\in W.$
\end{lemma}
\begin{proof}
	By Verdier duality, it suffices to show the statement for $\Delta_w.$ By Proposition \ref{prop:perverseexactness} we know that  $\Delta_w=j_{w,!}\un_{X_w}\in\DMT{(B)}{X}^{p\leq0}$. To show that $\Delta_w\in\DMT{(B)}{X}^{p\geq0}$ and is hence perverse we proceed by induction on the length of $w$. If $w=e$, then $j_w$ is a closed embedding. Hence $j_{w,!}=j_{w,*}$ and the statement follows from Proposition \ref{prop:perverseexactness}. Otherwise, let $s\in \Weyl$ be such that $ws<w$ and let $\pi:G/B\to G/P_s$ be the projection. Then we obtain the distinguished triangle
	\begin{equation}
\disttriangle{\Delta_{ws}}{\pi^!\pi_!\Delta_{ws}}{\Delta_{w}(1)[1]}
\end{equation}
	where $\Delta_{ws}$ is perverse by induction. Now let $x\in\Weyl$ and assume that $xs>s$. Then we obtain the cartesian square
	\begin{center}
\cartesiandiagramwithmaps{X_{xs}\cup X_x}{k}{X}{p}{\pi}{BxP_s/P_s}{i}{G/P_s.}
	\end{center}
	Applying base change and $k^*$ our distinguished triangle becomes
	\begin{equation}
\disttriangle{k^*\Delta_{ws}}{p^!p_!k^*\Delta_{ws}}{k^*\Delta_{w}(1)[1]}
\end{equation}
	Now $p$ is a trivial $\Proj^1$-bundle. Hence, we have reduced our statement to the case $X=\Proj^1$ where it follows easily.
\end{proof}
Hence Section \ref{subsec:perversemotives} implies the following equivalent descriptions of the category of stratified mixed Tate motives on $X$.
\begin{theorem}
	There are equivalences of categories
	\begin{align*}
	\Hotb{\Tilt(\Per{(B)}{X})}\stackrel{\sim}{\leftarrow}\DMT{(B)}{X}&\stackrel{\sim}{\rightarrow}\Hotb{\Projectives(\Per{(B)}{X})}\\&\stackrel{\sim}{\rightarrow}\Derb{\Per{(B)}{X})}.
	\end{align*}
\end{theorem} 
\section{Representation Theory} \label{sec:repTheory}
In this section we apply our results to the representation theory of semisimple algebraic groups in characteristic $p$.
\subsection{Modular Category $\cato$}
Let $G\supset B \supset T$ be a split semisimple simply connected algebraic group with a Borel subgroup and maximal torus over a field $\kk$ of characteristic $p$. 
Assume that $p$ is bigger than the Coxeter number of $G$. Denote by $N_G(T)/T=\Weyl\supset \simple$ the corresponding Weyl group and simple reflections and by $X(T)\supset \Phi\supset \Phi_+\supset\Delta$ the associated root lattice, root system, positive and simple roots.

For $\lambda\in X(T)$ dominant write $\Ind_B^G\kk_\lambda=H^0(G/B,\pazocal O(\lambda))=H^0(\lambda)$ for the induced representation of the one-dimensional $T$-module $\kk_\lambda$.
Over the complex numbers, those are exactly all simple rational representations (Borel--Weil--Bott Theorem) and they have a nice character formula (Weyl character formula). In positive characteristic though, the modules $H^0(\lambda)$ can become reducible and the main goal is to determine their composition factors. For astronomically big (see \cite{Fie}) prime numbers $p$, this is solved by the proof of the Lusztig conjecture in \cite{AJS}, which turns out to be false for smaller primes, as shown by \cite{Wil13} using the modular category $\cato$.

The \emph{modular category $\cato=\cato(G,B)$}, also called \emph{subquotient around the Steinberg point}, is a subquotient of the category $G\modules$ of finite dimensional representations of $G$ over $\kk$. It was defined by Soergel \cite{Soe90} in the following way. In the notation of \cite{Jan}, let $L(\lambda)$ denote the unique simple submodule of $\H^0(\lambda)$ and let  $\rho$ denote the half sum of all positive roots and $\operatorname{st}=(p-1)\rho$ the Steinberg weight. Soergel then defines two full subcategories of $G\modules$ by
\begin{align}
\pazocal{A}&=\{M\in G\modules\setline[M:L(\lambda)]\neq 0\Rightarrow\lambda \uparrow \operatorname{st}+\rho\}\text{ and}\\
\pazocal{N}&=\{M\in G\modules\setline[M:L(\lambda)]= 0\Rightarrow\lambda \in \operatorname{st}+\Weyl\rho\}.
\end{align}
Here $[M:L(\lambda)]$ is the number of times $L(\lambda)$ appears as factor in a composition series of $M$, and $\lambda \uparrow \operatorname{st}+\rho$ meams that $\lambda$ is linked to $\operatorname{st}+\rho$, with respect to the $p$-dilated action of the affine Weyl group (see \cite[Chapter 6]{Jan}).

\begin{definition} \label{defi:modO}
The modular category $\cato$ is then the Serre quotient 
\begin{equation}
\cato \stackrel{def}{=} \pazocal{A}/\pazocal{N}.
\end{equation}
\end{definition}

The modular category $\cato$ resembles the BGG category $\cato_0(\g)$ associated to \emph{complex} semisimple Lie algebras $\g$ (see \cite{BGG}) in many ways. It has standard objects $M_x=H^0(\operatorname{st}+x\rho)^*$ with unique simple quotient $L_x$ and projective covers $P_x$, all parametrized by elements of the Weyl group $x\in \Weyl$. In a way, it is a \emph{window into} or \emph{excerpt of} the category of all finite dimensional representations, which can be used to test or prove conjectures with methods used in the study of category $\cato_0(\g)$. Indeed, it was introduced by Soergel in \cite{Soe00} in the hope to partly prove the Lusztig conjecture and the mentioned counterexamples by Williamson \cite{Wil13} are constructed using the modular category $\cato$.

The modular category also has an analogue to the \emph{Struktur-} and \emph{Endomorphismensatz} from \cite{Soe90}.
\begin{theorem}[{\cite[19.8]{AJS}, \cite[Theorem 2.6.1]{Soe00}}]
The functor 
\begin{equation}
\V\stackrel{def}{=}\Hom{\cato}(P_{w_0}, - ): \cato\rightarrow \operatorname{mod-}\End{\cato}(P_{w_0}) 
\end{equation}
 is fully faithful on projective modules. Here $w_0$ denotes the longest element in $\Weyl.$ Furthermore
\begin{equation}
\End{\cato}(P_{w_0})=C\stackrel{def}{=}\SymA(\h)/\SymA(\h)_+^\Weyl.
\end{equation}
where $\h=\operatorname{Lie}(T)$, and $\SymA(\h)$ denotes the symmetric algebra.
\end{theorem}
Moreover, by analysing the interaction of the functor $\V$ with translation functors, \cite{Soe00} identifies the essential image of the projective modules in $\cato$ under $\V$ with the category $C\soergelmodules$ of \emph{Soergel modules}.
\begin{theorem}[\cite{Soe00} Theorem 2.8.2.]
	 The essential image of $\V$ is the category of \emph{Soergel modules} $$C\soergelmodules\stackrel{def}{=}\left\langle\,  C\otimes_{C^{s_1}}\cdots C\otimes_{C^{s_n}}\kk\,\vert\, s_i\in\Simp\,\right\rangle_{\oplus,\inplus,\cong}$$
	 where $\oplus$, $\inplus$ and $\cong$ means closure under finite direct sums, direct summands and isomorphisms in the category of $C$-modules.
\end{theorem}
Putting these results together we get a \emph{combinatorial} description of the derived modular category $\cato$ in terms of the homotopy category of Soergel modules
\begin{center}
	\begin{tikzcd}
	\Derb{\cato}\arrow[r,"\sim"']&\Hotb{\Projectives\cato}\arrow[r,"\sim"',"\V"]&\Hotb{C\soergelmodules}.
\end{tikzcd}
\end{center}
We now combine this with the results from Section \ref{subsec:flagvar}. Let $G^\vee\supset B^\vee\supset T^\vee$ be a semisimple algebraic group over $\overline{\F}_p$ with a Borel subgroup and maximal torus and root system $X(T^\vee)\supset\Phi^\vee$ dual to that of $G$.
Denote by $X^\vee=G^\vee/B^\vee$ the flag variety. Under the assumption that the torsion index $t_{\Phi^\vee},$ see Section \ref{subsubsec:parmotsoemod}, is invertible in $\kk$, Theorem \ref{thm:chowflag} gave us a description of the motivic cohomology ring of $X$ as
$$\Hyp(X) = 
\biggl (\SymA(X(T^\vee))/\SymA(X(T^\vee))_+^\Weyl \biggr ) \otimes\kk =
\SymA(\h)/\SymA(\h)_+^\Weyl=
C$$ 
and Theorem \ref{thm:mainflag} provided us with a combinatorial description of the category of stratified mixed Tate motives on $X^\vee$
\begin{center}
	\begin{tikzcd}
		\DMT{(B^\vee)}{X^\vee}\arrow[r,"\sim"']&\Hotb{\Par{(B^\vee)}{X^\vee}}\arrow[r,"\sim"',"\Hyp"]&\Hotb{C\soergelmodules^\Z_{ev}}.
	\end{tikzcd}
\end{center}
Putting everything together, we obtain our final theorem.
\begin{theorem}\label{thm:modularcategoryo} The functor induced by forgetting the grading of Soergel modules 
\begin{center}
	\begin{tikzcd}
		\DMT{(B^\vee)}{X^\vee,\kk}\arrow[r,"v"]&\Derb{\cato(G,B)}
	\end{tikzcd}
\end{center}
has the following properties:
\begin{enumerate}
	\item There is natural isomorphism $v \cong v\circ (1)[2]$.
	\item For all $E,F\in\DMT{(B^\vee)}{X^\vee,\kk}$ we can identify
	$$\bigoplus_{n\in\Z}\Hom{H(X^\vee)}(E,F(n)[2n])=\Hom{\Derb{\cato}}(v(E),v(F)).$$
	\item For every indecomposable projective module $P_x$ in $\cato$ there is an indecomposable pointwise pure parity motive $E_x\in\Par{(B^\vee)}{X^\vee}_{w=0}$ with $j^*_xE_x=\un_{X_s^\vee}$, such that $$v(E_x)=P_x.$$
	\item Costandard objects correspond to standard modules
	$$v(\nabla_x)=M_x.$$
\end{enumerate}
\end{theorem}
\begin{remark}
	We could also formulate the last theorem as an equivalence
		$$	\DMT{(B^\vee)}{G^\vee/B^\vee}\stackrel{\sim}{\rightarrow}\Derb{\cato^{\Z,ev}}$$
	as stated in the introduction, because the right hand side is by definition equivalent to $\Hotb{C\soergelmodules^{\Z,ev}}$. Alternatively, one could also artificially add a root of the Tate twist on the geometric side to get an equivalence with the whole derived graded category $\Derb{\cato^\Z}$.
\end{remark}
\appendix

\section{Rappels: Some category theory} \label{sec:Cat}

\subsection{Triangulated categories} \label{sec:triCat}

\begin{rappels} \label{rappels:tricat2}
Let $\Tt$ be a triangulated category.
\begin{enumerate}
 \item \label{VerdierLocalisation} Let $\Ss$ be a triangulated subcategory of $\Tt$. Then there is a description of the Verdier quotient $\Tt / \Ss$ whose objects are the same as those of  $\Tt$, and morphisms are equivalence classes of ``hats'' of morphisms $X \stackrel{s}{\leftarrow} \stackrel{f}{\to} Y$ with $f, s$ morphisms of $\Tt$ and $\Cone(s) \in \Ss$, cf. \cite[2.1.7, 2.2.1]{Ver96}, \cite[2.1.11]{Nee01}.

 \item \label{BousfieldRightadjoint}If $\Tt$ admits all small sums and is compactly generated, \cite[Def.1.7,1.8]{Nee96}, and $\Ss$ is a thick triangulated subcategory which also admits small sums and is generated by objects which are compact in $\Tt$, then the canonical quotient functor preserves small sums, and admits a right adjoint \cite[Exam.8.4.5]{Nee01}
\begin{equation}
 \Tt \rightleftarrows \Tt  / \Ss. 
\end{equation}

 \item \label{BousfieldLocalisationTriCat} Under the above hypotheses,  objects $\Ee$ such that $\Hom{\Tt}(f, \Ee)$ is an isomorphism for every $f$ with $\Cone(f)\in \Ss$, or equivalently, those objects $\Ee$ such that $\Hom{\Tt}(\Ff, \Ee) = 0$ for every $\Ff \in \Ss$, are called $\Ss$-\emph{local}. The right adjoint identifies $\Tt / \Ss$ with the full subcategory $\Ss$-local objectes in $\Tt$  \cite[Theo.9.1.16]{Nee01}. In symbols,
\begin{equation}
 \Tt  / \Ss \cong  \{ \Ee\in\Tt : \Ee \text{ is } \Ss\text{-local} \} \stackrel{\textrm{full}}{\subset} \Tt. 
\end{equation}

\end{enumerate}
\end{rappels}

\subsection{Descent model structures} \label{subsec:descentModelStru}
Recall the setup of Section \ref{subsec:DerCat}. In particular, $\R$ denotes a cartesian commutative monoïd of $\Sh_\Nis(\Sm / -)^\Sfrak$ in the monoïdal category of symmetric sequences in Nisnevich sheaves over $\Sm/ -.$
We recall here the definitions of cofibrations and fibrations in $C(\R\modules)$, even though in the text, we only use the definition of fibrant object, and the fact that the $t^n\R(X)[i{-}1]$ are cofibrant.

\begin{definition}[{\cite[5.1.11]{CD}}] \label{defi:cofibrations}
The class of \emph{cofibrations} of $C(\R\modules)$ is the smallest class of morphisms closed under retracts, pushouts, and transfinite composition, and containing the canonical morphisms 
\begin{equation}
t^n\R(X)[i] \to \Cone \biggl (\id: t^n\R(X)[i]{\to}t^n\R(X)[i] \biggr )
\end{equation}
 for all $X \in \Sm / S, i \in \Z, n \geq 0$. Note, the morphism $0 \to t^n\R(X)[i{+}1]$ is the pushout of this latter along $t^n\R(X)[i] \to 0$, so the objects $t^n\R(X)[i{+}1]$ are all cofibrant. We remark that the role of $i$ and $n$ are exchanged compared to \cite[5.1.11]{CD}.
\end{definition}

It follows from this that the functors $\R_S(X) \otimes -$ are left Quillen functors, but it seems that we never need this fact.

\begin{definition}[{\cite[Def.5.1.9, 5.1.11]{CD}}] \label{defi:fibrantKmodule}
An object $\Ee \in C(\R\modules)$ is \emph{fibrant} if and only if for every $i, n \geq 0$ and $X \in \Sm / S$, the (cochain complex) cohomology $H^i(X, \Ee_n)$ and the Nisnevich hypercohomology $\Hyp_{\Nis}^i(X, \Ee_n)$ agree.
\end{definition}

\begin{definition} \label{defi:fibrations}
A morphism $f: \Ee \to \Ff$ in $C(\R\modules)$ is a \emph{fibration} if each $\Ee_n \to \Ff_n$ is surjective as a morphism of \emph{presheaves}, and each kernel $\ker(\Ee_n \to \Ff_n)$ is fibrant.
\end{definition}

\bibliographystyle{amsalpha} 
\bibliography{main} 

\providecommand{\bysame}{\leavevmode\hbox to3em{\hrulefill}\thinspace}
\providecommand{\MR}{\relax\ifhmode\unskip\space\fi MR }
\providecommand{\MRhref}[2]{%
  \href{http://www.ams.org/mathscinet-getitem?mr=#1}{#2}
}
\providecommand{\href}[2]{#2}
\begin{thebibliography}{MVW06}

\bibitem[AGV71]{SGA43}
Michael Artin, Alexander Grothendieck, and Jean-Louis Verdier, \emph{Theorie de
  {T}opos et {C}ohomologie {\'e}tale des {S}chemas {III}}, Lecture Notes in
  Mathematics, vol. 305, Springer, 1971.

\bibitem[AJS94]{AJS}
H.~H. Andersen, J.~C. Jantzen, and W.~Soergel, \emph{Representations of quantum
  groups at a {$p$}th root of unity and of semisimple groups in characteristic
  {$p$}: independence of {$p$}}, Ast\'erisque (1994), no.~220, 321.
  \MR{1272539}

\bibitem[AR16a]{AR1}
P.~N. Achar and S.~Riche, \emph{Modular perverse sheaves on flag varieties {I}:
  tilting and parity sheaves}, Ann. Sci. \'Ec. Norm. Sup\'er. (4) \textbf{49}
  (2016), no.~2, 325--370, With a joint appendix with Geordie Williamson.
  \MR{3481352}

\bibitem[AR16b]{AR2}
\bysame, \emph{Modular perverse sheaves on flag varieties, {II}: Koszul duality
  and formality}, Duke Math. J. \textbf{165} (2016), no.~1, 161--215.

\bibitem[AR16c]{AR16}
\bysame, \emph{{Reductive groups, the loop Grassmannian, and the Springer
  resolution}}, ArXiv e-prints (2016).

\bibitem[Ayo07]{Ayo07}
Joseph Ayoub, \emph{Les six op\'erations de {G}rothendieck et le formalisme des
  cycles \'evanescents dans le monde motivique. {I}}, Ast\'erisque (2007),
  no.~314, x+466 pp. (2008). \MR{2423375}

\bibitem[BBD82]{BBD}
A.~Beilinson, J.~Bernstein, and P.~Deligne, \emph{Faisceaux pervers}, Analysis
  and topology on singular spaces, {I} ({L}uminy, 1981), Ast\'erisque, vol.
  100, Soc. Math. France, Paris, 1982, pp.~5--171. \MR{751966}

\bibitem[BBM04]{BBM}
A.~Beilinson, R.~Bezrukavnikov, and I.~Mirkovi{\'c}, \emph{Tilting exercises},
  Mosc. Math. J. \textbf{4} (2004), no.~3, 547--557, 782. \MR{2119139}

\bibitem[BGG71]{BGG}
I.~N. Bernstein, I.~M. Gelfand, and S.~I. Gelfand, \emph{Structure of
  representations that are generated by vectors of highest weight},
  Funckcional. Anal. i Prilo\v zen. \textbf{5} (1971), no.~1, 1--9. \MR{0291204
  (45 \#298)}

\bibitem[BGS96]{BGS}
Alexander Beilinson, Victor Ginzburg, and Wolfgang Soergel, \emph{Koszul
  duality patterns in representation theory}, J. Amer. Math. Soc. \textbf{9}
  (1996), no.~2, 473--527. \MR{1322847}

\bibitem[Blo86]{Blo86}
Spencer Bloch, \emph{Algebraic cycles and higher {$K$}-theory}, Adv. in Math.
  \textbf{61} (1986), no.~3, 267--304. \MR{852815}

\bibitem[Bon10]{Bon}
Mikhail~V. Bondarko, \emph{Weight structures vs. {$t$}-structures; weight
  filtrations, spectral sequences, and complexes (for motives and in general)},
  J. K-Theory \textbf{6} (2010), no.~3, 387--504. \MR{2746283}

\bibitem[Bon14]{Bon14}
\bysame, \emph{Weights for relative motives: Relation with mixed complexes of
  sheaves}, International Mathematics Research Notices \textbf{2014} (2014),
  no.~17, 4715--4767.

\bibitem[CD09]{CDHomAlg}
Denis-Charles Cisinski and Fr{\'e}d{\'e}ric D{\'e}glise, \emph{Local and stable
  homological algebra in {G}rothendieck abelian categories}, Homology, Homotopy
  Appl. \textbf{11} (2009), no.~1, 219--260. \MR{2529161}

\bibitem[CD12]{CD}
\bysame, \emph{{Triangulated categories of mixed motives}}, Arxiv preprint
  arXiv:0912.2110 (2012).

\bibitem[CYZ08]{doi:10.1080/00927870701649184}
Xiao-Wu Chen, Yu~Ye, and Pu~Zhang, \emph{Algebras of derived dimension zero},
  Communications in Algebra \textbf{36} (2008), no.~1, 1--10.

\bibitem[dCM09]{CM}
Mark Andrea~A. de~Cataldo and Luca Migliorini, \emph{The decomposition theorem,
  perverse sheaves and the topology of algebraic maps}, Bull. Amer. Math. Soc.
  (N.S.) \textbf{46} (2009), no.~4, 535--633. \MR{2525735}

\bibitem[Del]{Del01}
Pierre Deligne, \emph{Voevodsky{\textquoteright}s lectures on cross functors,
  {F}all 2001}, Unpublished.

\bibitem[Dem73]{Dem73}
Michel Demazure, \emph{Invariants symmétriques entiers des groupes de {W}eyl
  et torsion}, Inventiones mathematicae \textbf{21} (1973), 287--302.

\bibitem[Dem74]{Dem74}
\bysame, \emph{Désingularisation des variétés de {S}chubert
  généralisées}, Annales scientifiques de l'École Normale Supérieure
  \textbf{7} (1974), no.~1, 53--88 (fre).

\bibitem[Fie12]{Fie}
Peter Fiebig, \emph{An upper bound on the exceptional characteristics for
  {L}usztig's character formula}, J. Reine Angew. Math. \textbf{673} (2012),
  1--31. \MR{2999126}

\bibitem[Gin91]{Gin}
Victor Ginsburg, \emph{Perverse sheaves and {${\bf C}\sp *$}-actions}, J. Amer.
  Math. Soc. \textbf{4} (1991), no.~3, 483--490. \MR{1091465}

\bibitem[GL00]{GL}
Thomas Geisser and Marc Levine, \emph{The {$K$}-theory of fields in
  characteristic {$p$}}, Invent. Math. \textbf{139} (2000), no.~3, 459--493.
  \MR{1738056}

\bibitem[GM88]{GM88}
Mark Goresky and Robert MacPherson, \emph{Stratified {M}orse theory},
  Ergebnisse der Mathematik und ihrer Grenzgebiete (3) [Results in Mathematics
  and Related Areas (3)], vol.~14, Springer-Verlag, Berlin, 1988. \MR{932724}

\bibitem[Gro58]{Gro58}
Alexandre Grothendieck, \emph{Torsion homologique et sections rationnelles},
  Séminaire Claude Chevalley \textbf{3} (1958), 1--29 (fre).

\bibitem[Har77]{Har}
Robin Hartshorne, \emph{Algebraic geometry}, Springer-Verlag, New York, 1977,
  Graduate Texts in Mathematics, No. 52. \MR{MR0463157 (57 \#3116)}

\bibitem[HSS00]{HSS}
Mark Hovey, Brooke Shipley, and Jeff Smith, \emph{Symmetric spectra}, J. Amer.
  Math. Soc. \textbf{13} (2000), no.~1, 149--208. \MR{1695653}

\bibitem[Jan03]{Jan}
Jens~Carsten Jantzen, \emph{Representations of algebraic groups}, second ed.,
  Mathematical Surveys and Monographs, vol. 107, American Mathematical Society,
  Providence, RI, 2003. \MR{2015057}

\bibitem[JMW14]{JMW}
Daniel Juteau, Carl Mautner, and Geordie Williamson, \emph{Parity sheaves}, J.
  Amer. Math. Soc. \textbf{27} (2014), no.~4, 1169--1212. \MR{3230821}

\bibitem[Kel93]{Kel93}
Bernhard Keller, \emph{A remark on tilting theory and dg algebras}, manuscripta
  mathematica \textbf{79} (1993), no.~1, 247--252.

\bibitem[Kri13]{Kri13}
Amalendu Krishna, \emph{Higher {C}how groups of varieties with group action},
  Algebra \& Number Theory (2013), 449.

\bibitem[KS86]{KaSa}
Kazuya Kato and Shuji Saito, \emph{Global class field theory of arithmetic
  schemes}, Applications of algebraic {$K$}-theory to algebraic geometry and
  number theory, {P}art {I}, {II} ({B}oulder, {C}olo., 1983), Contemp. Math.,
  vol.~55, Amer. Math. Soc., Providence, RI, 1986, pp.~255--331. \MR{862639}

\bibitem[LC07]{LE2007452}
Jue Le and Xiao-Wu Chen, \emph{Karoubianness of a triangulated category},
  Journal of Algebra \textbf{310} (2007), no.~1, 452 -- 457.

\bibitem[Lev94]{Lev94}
Marc Levine, \emph{Bloch's higher {C}how groups revisited}, Ast\'erisque
  (1994), no.~226, 10, 235--320, $K$-theory (Strasbourg, 1992). \MR{1317122}

\bibitem[Lus10]{lusztig2010introduction}
G.~Lusztig, \emph{Introduction to quantum groups}, Modern Birkh{\"a}user
  Classics, Birkh{\"a}user Boston, 2010.

\bibitem[Mil70]{Mil69}
John Milnor, \emph{Algebraic {$K$}-theory and quadratic forms}, Invent. Math.
  \textbf{9} (1969/1970), 318--344. \MR{0260844}

\bibitem[MV99]{MV99}
Fabien Morel and Vladimir Voevodsky, \emph{{${\bf A}^1$}-homotopy theory of
  schemes}, Inst. Hautes \'Etudes Sci. Publ. Math. (1999), no.~90, 45--143
  (2001). \MR{1813224}

\bibitem[MVW06]{MVW}
Carlo Mazza, Vladimir Voevodsky, and Charles Weibel, \emph{Lecture notes on
  motivic cohomology}, Clay Mathematics Monographs, vol.~2, American
  Mathematical Society, Providence, RI; Clay Mathematics Institute, Cambridge,
  MA, 2006. \MR{2242284}

\bibitem[Nee96]{Nee96}
Amnon Neeman, \emph{The {G}rothendieck duality theorem via {B}ousfield's
  techniques and {B}rown representability}, J. Amer. Math. Soc. \textbf{9}
  (1996), no.~1, 205--236. \MR{1308405}

\bibitem[Nee01]{Nee01}
\bysame, \emph{Triangulated categories}, Annals of Mathematics Studies, vol.
  148, Princeton University Press, Princeton, NJ, 2001. \MR{1812507
  (2001k:18010)}

\bibitem[Ric89]{Ric89}
Jeremy Rickard, \emph{Morita theory for derived categories}, Journal of the
  London Mathematical Society \textbf{2} (1989), no.~3, 436--456.

\bibitem[Rin91]{Rin}
Claus~Michael Ringel, \emph{The category of modules with good filtrations over
  a quasi-hereditary algebra has almost split sequences.}, Mathematische
  Zeitschrift \textbf{208} (1991), no.~2, 209--224.

\bibitem[RSW14]{RSW14}
Simon Riche, Wolfgang Soergel, and Geordie Williamson, \emph{Modular {K}oszul
  duality}, Compositio Mathematica \textbf{150} (2014), 273--332.

\bibitem[Sai89]{Saito}
Morihiko Saito, \emph{Introduction to mixed {H}odge modules}, Ast\'erisque
  (1989), no.~179-180, 10, 145--162, Actes du Colloque de Th{\'e}orie de Hodge
  (Luminy, 1987). \MR{1042805}

\bibitem[SGA03]{SGA1}
\emph{Rev\^etements \'etales et groupe fondamental ({SGA} 1)}, Documents
  Math\'ematiques (Paris) [Mathematical Documents (Paris)], 3, Soci\'et\'e
  Math\'ematique de France, Paris, 2003, S{\'e}minaire de g{\'e}om{\'e}trie
  alg{\'e}brique du Bois Marie 1960--61. [Algebraic Geometry Seminar of Bois
  Marie 1960-61], Directed by A. Grothendieck, With two papers by M. Raynaud,
  Updated and annotated reprint of the 1971 original [Lecture Notes in Math.,
  224, Springer, Berlin; MR0354651 (50 \#7129)]. \MR{2017446 (2004g:14017)}

\bibitem[Soe90]{Soe90}
Wolfgang Soergel, \emph{Kategorie {$\pazocal O$}, perverse {G}arben und
  {M}oduln \"uber den {K}oinvarianten zur {W}eylgruppe}, J. Amer. Math. Soc.
  \textbf{3} (1990), no.~2, 421--445. \MR{1029692}

\bibitem[Soe00]{Soe00}
\bysame, \emph{On the relation between intersection cohomology and
  representation theory in positive characteristic}, J. Pure Appl. Algebra
  \textbf{152} (2000), no.~1-3, 311--335, Commutative algebra, homological
  algebra and representation theory (Catania/Genoa/Rome, 1998). \MR{1784005}

\bibitem[Spr82]{Springer}
T.~A. Springer, \emph{Quelques applications de la cohomologie d'intersection},
  Séminaire Bourbaki \textbf{24} (1981-1982), 249--273 (fre).

\bibitem[SW16]{SoeWe}
Wolfgang Soergel and Matthias Wendt, \emph{Perverse motives and graded derived
  category $\pazocal{O}$}, Journal of the Institute of Mathematics of Jussieu
  (2016), 1--49.

\bibitem[Ver96]{Ver96}
Jean-Louis Verdier, \emph{Des cat\'egories d\'eriv\'ees des cat\'egories
  ab\'eliennes}, Ast\'erisque (1996), no.~239, xii+253 pp. (1997), With a
  preface by Luc Illusie, Edited and with a note by Georges Maltsiniotis.
  \MR{1453167}

\bibitem[Voe96]{Voev96}
Vladimir Voevodsky, \emph{Homology of schemes}, Selecta Math. (N.S.) \textbf{2}
  (1996), no.~1, 111--153. \MR{1403354 (98c:14016)}

\bibitem[Voe00]{Voev00}
\bysame, \emph{Triangulated categories of motives over a field}, Cycles,
  transfers, and motivic homology theories, Ann. of Math. Stud., vol. 143,
  Princeton Univ. Press, Princeton, NJ, 2000, pp.~188--238. \MR{1764202}

\bibitem[{Wil}13]{Wil13}
G.~{Williamson}, \emph{{Schubert calculus and torsion explosion}}, ArXiv
  e-prints (2013).

\bibitem[WW17]{webster2017}
Benjamin Webster and Geordie Williamson, \emph{A geometric construction of
  colored homflypt homology}, Geom. Topol. \textbf{21} (2017), no.~5,
  2557--2600.

\end{thebibliography}
\end{document}